\documentclass[12pt,reqno]{amsart}

\usepackage[usenames]{color}
\usepackage{amsmath,pdfsync,verbatim,graphicx,epstopdf,enumerate}
\usepackage[colorlinks=true]{hyperref}
\hypersetup{allcolors=blue}
\numberwithin{equation}{section}

\newcommand{\I}{{i}}
\newcommand{\D}{\mathrm{d}}

\newcommand{\wh}{\widehat}

\newcommand{\lb}{\left(}
\newcommand{\Lm}{\left\lvert}
\newcommand{\Rm}{\right\rvert}
\newcommand{\vp}{\varphi}
\newcommand{\ve}{\varepsilon}

\newcommand{\rb}{\right)}

\newcommand{\wt}{\widetilde}
\newcommand{\Ac}{\mathcal{A}}
\newcommand{\Bc}{\mathcal{B}}

\newcommand{\Dc}{\mathcal{D}}
\newcommand{\Ec}{\mathcal{E}}

\newcommand{\Nc}{\mathcal{N}}
\newcommand{\Oc}{\mathcal{O}}

\newcommand{\Qc}{\mathcal{Q}}

\newcommand{\Wc}{\mathcal{W}}

\newcommand{\Bb}{\mathbb{B}}
\newcommand{\Cb}{\mathbb{C}}

\newcommand{\Rb}{\mathbb{R}}

\newcommand{\Beq}{\begin{equation}}
\newcommand{\Eeq}{\end{equation}}
\newcommand{\beq}{\begin{equation*}}
\newcommand{\eeq}{\end{equation*}}
\newcommand{\bal}{\begin{align}}
\newcommand{\eal}{\end{align}}
\renewcommand{\O}{\Omega}

\newcommand{\LM}{\left\lVert}
\newcommand{\RM}{\right\rVert}

\newcommand{\A}{\alpha}

\newcommand{\bp}{\begin{prob}}
\newcommand{\ep}{\end{prob}}
\newcommand{\bpr}{\begin{proof}}
\newcommand{\epr}{\end{proof}}

\newcommand{\Lc}{\mathcal{L}}

\newcommand{\tred}[1]{{\color{red}{#1}}}

\newcommand{\st}{\,:\,}

\newcommand{\bel}[1]{\begin{equation}\label{#1}}
\newcommand{\ee}{\end{equation}}

\newcommand{\ssubset}{\subset\joinrel\subset}

\newtheorem{theorem}{Theorem}[section]
\newtheorem{corollary}[theorem]{Corollary}

\newtheorem{proposition}[theorem]{Proposition}

\newtheorem{definition}[theorem]{Definition}
\newtheorem{remark}[theorem]{Remark}

\newcommand{\h}{\tau}
\newcommand{\hh}{\tau_1}

\title[Inverse Problem for Biharmonic Operator]{An Inverse Problem on Determining Second order Symmetric Tensor for Perturbed Biharmonic Operator}
\author[Bhattacharyya and Ghosh]{Sombuddha Bhattacharyya$^\ast$ and Tuhin Ghosh$^{\ast\ast}$}
\address{$^{\ast}$ Department of Mathematics,
\newline\indent\: \ Indian Institute of Science Education and Research, Bhopal, India.
\newline
\indent\: \
E-mail:{\tt \  sombuddha@iiserb.ac.in}}
\address {$^{\ast\ast}$Department of Mathematics, 
	Universit\"{a}t Bielefeld.
\newline
\indent\: 
E-mail:{\tt\  tghosh@math.uni-bielefeld.de}}

\begin{document}

\begin{abstract}
This article offers a study of the Calder\'{o}n type inverse problem of determining up to second order coefficients of higher order elliptic operators. Here we show that it is possible to determine an anisotropic second order perturbation given by a symmetric matrix, along with a first order perturbation given by a vector field and a zero-th order potential function inside a bounded domain, by measuring the Dirichlet to Neumann map of the perturbed biharmonic operator on the boundary of that domain. 



\end{abstract}

\maketitle
\section{Introduction and statement of the main result}
Let $\Omega\subset\Rb^{n},$ $n\geq 3$ be a bounded domain with smooth connected boundary. Let us consider the following perturbed biharmonic operator $\Lc(x,D)$ of order $4$, with perturbations up to second order, of the following form: 
\begin{equation}\label{operator}
	\Lc_{A,B,q}(x,D)=\Lc(x,D):=
	(-\Delta)^2 + \sum_{j,k=1}^{n} A_{jk}(x) D_{j}D_{k} + \sum_{j=1}^{n} B_{j}(x) D_{j} + q(x),
\end{equation}
where $D_j=\frac{1}{i}\partial_{x_j}$, $A=(A_{jk}) \in W^{3,\infty}(\Omega, \mathbb{C}^{n^2})$, $B=(B_{j}) \in W^{2,\infty}(\Omega, \mathbb{C}^{n})$ and $q \in L^{\infty}(\Omega, \mathbb{C})$ are the perturbation coefficients.

Here in this article we will be considering an inverse problem of recovering the coefficients $A$, $B$, $q$ in $\Omega$ from the knowledge of the Dirichlet to Neumann map (DN map):
\[\Big{(} u|_{\partial\Omega}, (-\Delta)u|_{\partial\Omega} \Big{)}\mapsto  \Big{(}\partial_{\nu}u|_{\partial\Omega}, \partial_{\nu}(-\Delta)u|_{\partial\Omega}\Big{)}\] 
corresponding to $\Lc(x,D)$, given on the boundary $\partial\O$.
The novelty of this work lies in full global recovery of the second order anisotropic matrix perturbation, along with the first and zero-th order perturbed terms of a biharmonic operator.

The inverse problem we present in this article falls into the generalized  category of the Calder\'{o}n type inverse problem. The original Calder\'on's problem \cite{Calderon1980} modeled by the second order elliptic operator $-\nabla\cdot\gamma\nabla$, first  appeared on studying the Electrical Impedance Tomography (EIT), where one uses static voltage and current measurements at the boundary of an object to know about its internal conductivity $\gamma$. If we assume the conductivity is isotropic and regular, then we do recover it from the voltage current measurements at the boundary. See the seminal work of Sylvester and Uhlmann \cite{SU} in dimension three and higher, and Nachmann \cite{NA} for the two dimensional case. If the conductivity is anisotropic, then the unique recovery assertion fails \cite{SUa}. The inverse boundary value problem for magnetic Schr\"odinger operator $(D+A)^2+q$ that on the question of  recovering the vector field $A$ appearing as a first order perturbation of the Laplacian operator, \cite{SUN,DOS} show that the recovery of the magnetic field $A$ along with unique recovery of scalar potential $q$ can be achieved up to its natural gauge invariance, $A=A+\nabla\varphi$ for $\varphi$ scalar.

A generalization of Calder\'{o}n type inverse problems for higher order (order $>2$) operators gets its due attention, and in this direction one seeks to recover lower order perturbations of a higher order operator from the boundary DN map. The work of \cite{KRU1, KRU2} first offer such study, and establish the complete recovery of the first (vector field) and zeroth order (scalar potential) coefficients of perturbed higher order elliptic operator of order $2m$:
\begin{equation}\label{l1} \Lc(x,D) = (-\Delta)^m  + \sum_{j=1}^{n} B_{j}(x) D_{j} + q(x), \quad m\geq 2.\end{equation}
Note that, the complete recovery of the first order perturbation fails for $m=1$ case. Following that, subsequent variations of the problems has been addressed, see \cite{IKE,TS,YANG,AYY,AY,KU,YK} to mention. However, all these works remained considering up to first order perturbations (cf. \eqref{l1}) of the biharmonic or polyharmonic operators. 

One interesting aspect remained for example by allowing up to the second order perturbation of a biharmonic or polyharmonic operator, and seeking the determination of all the coefficients. Previously, \cite{VT} carries out one such exercise to establish that if the second order perturbation is governed by an isotropic matrix i.e. $A_{jk}=a(x)\delta_{jk}$ for some scalar function $a(x)$ in 
\begin{equation}\label{l2}
	\Lc(x,D) = (-\Delta)^m + a(x)(-\Delta) + \sum_{j=1}^{n} B_{j}(x) D_{j} + q(x), \quad m\geq2,
\end{equation}
 then $a$ can be determined along with the $B$, $q$ by knowing the boundary DN map. Then in \cite{BG19} we extend this result by showing that it is possible to determine any symmetric matrix $A$ from the perturbed polyharmonic operator:
 \begin{equation}\label{l3} \Lc(x,D) = (-\Delta)^m + \sum_{j,k=1}^{n} A_{jk}(x) D_{j}D_{k} + \sum_{j=1}^{n} B_{j}(x) D_{j} + q(x), \quad m>2.
 \end{equation}
 Clearly, the compromise we made there in order to recover anisotropic matrices $A$, we sought $A_{jk}D_jD_k$ as the second order perturbation of the polyharmonic operator of order at least $6$ ($m=3$).     So the border line case for the perturbed biharmonic operator ($m=2$) remained open. In this article, we settle that part by establishing the recovery of the anisotropic matrix $A$ of the perturbed biharmonic operator ($m=2$) (cf. \eqref{operator}) from the knowledge of its boundary DN map.  
  The presence of the anisotropic matrix $A$ in $\Lc(x,D)$ (cf. \eqref{operator}) brings a number of challenges, we will get into that in the following sections. 

 The higher order elliptic operators as in \eqref{operator} arise in the areas of physics and geometry, such as the study of the Kirchoff plate equation (perturbed biharmonic operator) in the theory of elasticity, buckling problem and the study of the Paneitz-Branson operator in conformal geometry, for more details see \cite{GGS, Ashbaugh2004}. For more on the elasticity model and perturbed biharmonic operators see \cite{Selvadurai2000,NU,Luis-Campos}.
 A related study of unique continuation for Kirchoff-Love plate equation or in general fourth-order elliptic equation has its own appeal, we refer \cite{CFGC,CFKH,LNW,SHTA} and reference therein. Let us also mention the recent work of \cite{ARV} which studies the boundary unique continuation results for the Kirchoff-Love plate equation. The main method in both unique continuation and inverse problems remain to perform Carleman method of estimations, which we have discussed in Section \ref{Sec_amplitude}. We end our discussion here with mentioning this recent publication \cite{LJRL} for a detailed understanding and wide references on the related topics of biharmonic operator.

\subsection{Direct Problem}
Let $\Omega\subset\Rb^{n},$ $n\geq 3$ be a bounded domain with smooth connected boundary. 
Recall the operator $\Lc(x,D)$ given as \eqref{operator}, defined in $\O$.
Consider the domain of this operator to be 
\[
\Dc(\Lc(x,D)) = \Big{\{}u\in H^{4}(\Omega); u|_{\partial\Omega}=0=(-\Delta)u|_{\partial\Omega} \Big{\}}.
\]
The operator $\Lc(x,D)$ in the domain $\Dc(\Lc(x,D))$ is an unbounded closed operator with a purely discrete spectrum \cite{Gerd_Grubb}.
We make the assumption that $0$ is not an eigenvalue of the operator $\Lc(x,D) : \Dc(\Lc(x,D)) \to L^2(\O)$. Let us denote
\[ \gamma u = \Big{(} u|_{\partial\Omega}, (-\Delta)u|_{\partial\Omega} \Big{)},\]
then for any $f=(f_0,f_1) \in H^{\frac{7}{2}}(\partial\Omega)\times H^{\frac{3}{2}}(\partial\Omega)$, the boundary value problem, 
\begin{equation}\begin{aligned}\label{problem}
\begin{cases}
\Lc(x,D)u &= 0\quad \mbox{ in }\Omega, \\
\gamma u &= f \quad \mbox{ on }\partial\Omega,
\end{cases}
\end{aligned}
\end{equation}
has a unique solution $u_f \in H^{4}(\Omega)$.

Let us define the corresponding Neumann trace operator $\gamma^{\#}$ by
\[
\gamma^{\#} u = \Big{(}\partial_{\nu}u|_{\partial\Omega}, \partial_{\nu}(-\Delta)u|_{\partial\Omega}\Big{)},
\]
where $\nu$ is the outward unit normal to the boundary $\partial\Omega$.
The Dirichlet to Neumann map (DN map) corresponding to the operator $\Lc(x,D)$ is given as
\begin{equation}\label{dnmap}\begin{aligned}
&\Nc : H^{\frac{7}{2}}(\partial\Omega)\times H^{\frac{3}{2}}(\partial\Omega) \to H^{\frac{5}{2}}(\partial\Omega)\times H^{\frac{1}{2}}(\partial\Omega),\\
&\Nc(f) = \gamma^{\#} u_f = \Big{(}\partial_{\nu}u_f|_{\partial\Omega}, \partial_{\nu}(-\Delta)u_f|_{\partial\Omega}\Big{)},
\end{aligned}
\end{equation}
where $u_f\in H^{4}(\Omega)$ is the unique solution of \eqref{problem}. We also define the Cauchy data set to be the graph of the DN map as
\begin{equation}\label{qw19}
\mathcal{C}^N= \{\Big{(}u|_{\partial\Omega}, (-\Delta)u|_{\partial\Omega}\ ;\  \partial_{\nu}u|_{\partial\Omega}, \partial_{\nu}(-\Delta)u|_{\partial\Omega} \Big{)} \st \Lc(x,D)u=0, \mbox{ in }\Omega \}.
\end{equation}

Before we move into addressing the inverse problem, here we quickly go through this following observation.  
From the wellposedness of the direct problem we readily see that knowing the coefficients $A,B,q$ in $\O$, one can determine the DN map on $\partial\O$. Therefore, we can establish a one sided relation:
\[\mathfrak{T}_{\Lc}: \mbox{ Knowledge of the coefficients in }\O \quad \longrightarrow\quad
\mbox{Dirichlet to Neumann map }\mathcal{N} \mbox{ on }\partial\O. 
\]
Our goal is to address the injectivity of the above mapping. But first let us consider a general biharmonic operator $\mathcal{M}$ with lower order perturbation up to order 3, defined as
\[	\mathcal{M}_{C,A,B,q}(x,D) = (-\Delta)^2 + \sum_{j,k,l=1}^{n} C_{jkl}\frac{\partial^3}{\partial x_j \partial x_k \partial x_l} + \sum_{j,k=1}^{n} A_{jk}\frac{\partial^2}{\partial x_j \partial x_k} + \sum_{j=1}^{n} B_j\frac{\partial}{\partial x_j} + q,
\]
where $C(x)$ is a symmetric 3-tensor, $A(x)$ is a symmetric matrix, $B(x)$ is a vector field and $q(x)$ is a function and $C,A,B,q$ are smooth in $\overline{\Omega}$.
Let $u \in H^4(\Omega)$ be a solution of $\mathcal{M}_{C,A,B,q} u = 0$ in $\Omega$, then for any $\Phi \in H^{4}_0(\Omega)$ we see that
\begin{equation}\label{invariance_1}
\mathcal{M}_{\mathcal{C},\mathcal{A},\mathcal{B},\mathcal{Q}} \left(ue^{\Phi}\right) = 0, \qquad \mbox{in }\Omega,
\end{equation}
where $\Wc,\Ac,\Bc,\Qc$ are a set of $C^{\infty}(\overline{\Omega})$ coefficients given as
\begin{equation}\label{Gauge_1}
\begin{aligned}
	\mathcal{C} =& C - \nabla\Phi\otimes I,\\
	\mathcal{A} =& A - 4\left(\nabla\Phi\otimes\nabla\Phi\right) - 4\nabla^2\Phi - \left(|\nabla\Phi|^2 + \Delta\Phi\right)I - 3\langle C,\nabla\Phi \rangle,\\
	\mathcal{B} =& B - 6(\Delta \Phi)\nabla\Phi - 4\nabla(\Delta\Phi) - 8\nabla\Phi\left(\nabla\otimes\nabla\Phi\right) - 2|\nabla \Phi|^2\nabla\Phi - 2(\nabla\Phi A)\\
	&\quad -3\langle C, \nabla^2\Phi\rangle - 3\langle C,\nabla\Phi \otimes\nabla\Phi \rangle,\\
	\mathcal{Q} =& q - (\Delta \Phi)^2 - 2|\nabla\Phi|^2(\Delta\Phi) - 4\nabla\Phi \cdot (\nabla\Delta\Phi) - \Delta^2\Phi - 2(\nabla^2\Phi : \nabla^2\Phi)\\
	&\quad-4\nabla\Phi\left(\nabla\otimes\nabla\Phi\right)\nabla\Phi - |\nabla\Phi|^4 -\langle C,\nabla^3\Phi \rangle - 3\langle C,\nabla^2\Phi\otimes\nabla \Phi \rangle \\
	&\quad -\langle C,\nabla\Phi\otimes \nabla\Phi\otimes \nabla\Phi\rangle - (A:\nabla^2\Phi) - (A\nabla\Phi)\cdot\nabla\Phi - (B\cdot\nabla\Phi).
\end{aligned}
\end{equation}
Note that
\[	\left(ue^{\Phi},\partial_{\nu} (ue^{\Phi})\, ;\, \partial_{\nu}^{2} (ue^{\Phi}), \partial_{\nu}^{3} (ue^{\Phi})\right)|_{\partial\Omega}=\left(u,\partial_{\nu} u\, ;\, \partial_{\nu}^{2}u, \partial_{\nu}^{3} u\right)|_{\partial\Omega},
\]
which implies $\mathcal{M}_{C,A,B,q}$ and $\mathcal{M}_{\mathcal{C},\mathcal{A},\mathcal{B},\mathcal{Q}}$ has the same DN map where $\mathcal{C}, \mathcal{A}, \mathcal{B},\mathcal{Q}$ are as in \eqref{Gauge_1} and $\Phi \in H^4_0(\Omega)$.

This shows that $\mathfrak{T}_{\mathcal{M}}$ is certainly not injective for the operator $\mathcal{M}$. This kind of obstruction towards injectivity exists for the magnetic Schr\"odinger operator as well, see \cite{SUN,DOS}.
Since $\Lc$ is a special case of the operator $\mathcal{M}$, i.e. $\mathcal{M}=\Lc$ when $C=0$ in $\O$, therefore one might doubt about the injectivity of $\mathfrak{T}_{\Lc}$.
	
Now, for the special case if $C = 0$ in $\O$, we get $\mathcal{M}_{0,A,B,q} = \Lc_{A,B,q}(x,D)$ in $\O$. If there is a gauge for the operator $\Lc(x,D)$, then it would mean that there exist an operator $\Lc_{\mathcal{A},\mathcal{B},\mathcal{Q}} = \mathcal{M}_{0,\mathcal{A},\mathcal{B},\mathcal{Q}}$ such that it satisfies the relation in \eqref{Gauge_1} for some $\Phi \in H^4_0(\O)$ such that $\mathcal{C}=0$ in $\O$. Since, we have assumed $C=0$ in $\O$, this means $\nabla\Phi\otimes I=0$ in $\O$ for $\Phi \in H^4_0(\O)$, which implies $\Phi=0$ in $\O$. So the absence of the third order perturbations in $\Lc(x,D)$ sets up the possibility of the complete recovery of $A$, $B$ and $q$ in $\O$ from the knowledge of the DN map.

\subsection{Inverse problem}
The \textit{inverse problem} we investigate here is, does the DN map $\mathcal{N}$ determines the unknown coefficients of $\Lc(x,D)$, namely the symmetric matrix $A$ along with the vector field $B$ and the potential function $q$ in $\O$?

In this article we provide an affirmative answer to this question. Let $\wt{A}\in W^{3,\infty}(\Omega)$ be a symmetric matrix, $\wt{B} \in W^{2,\infty}(\Omega)$ be a vector field and $q\in L^{\infty}(\Omega)$ be a function and write $\wt{\Lc}(x,D) = \Lc_{\wt{A},\wt{B},\wt{q}}(x,D)$ defined in $\Omega$. Let $\wt{\mathcal{N}}$ be the DN map corresponding to the operator $\wt{\Lc}(x,D)$. 

Let $\Ec'(\overline{\Omega})$ be the dual of $\Ec(\overline{\Omega}) = C^{\infty}(\overline{\Omega})$. In particular, $\Ec'(\overline{\Omega})$ is the space of all compactly supported distributions in $\overline{\Omega}$.
We state our main result here.
\begin{theorem}\label{mainresult}
	Let $\Omega\subset \mathbb{R}^n, n\geq 3 $ be a bounded domain with smooth connected boundary.
	Let $\Lc(x,D)$ and $\wt{\Lc}(x,D)$ be two operators defined as in \eqref{operator} with the coefficients $A,\wt{A} \in W^{3,\infty}(\Rb^n;\Cb^{n^2}) \cap \mathcal{E}^{\prime}(\overline{\Omega})$; $B,\wt{B} \in W^{2,\infty}(\Rb^n;\Cb^{n}) \cap \mathcal{E}^{\prime}(\overline{\Omega})$ and $q,\widetilde{q} \in L^{\infty}(\Omega;\Cb)$.
	Assume that $0$ is not an eigenvalue for $\Lc(x,D), \wt{\Lc}(x,D)$ in $\Dc(\Lc(x,D))$ and $\Dc(\wt{\Lc}(x,D))$ respectively.
	If 
	\begin{equation*}\label{Neumann Data}
		\begin{aligned}
			&\Nc(f)|_{\partial\Omega} = \widetilde{\Nc}(f)|_{\partial\Omega} \quad \mbox{ for all } f \in H^{\frac{7}{2}}(\partial\Omega)\times H^{\frac{3}{2}}(\partial\Omega),
		\end{aligned}
	\end{equation*}
	then 
	\[
	A=\widetilde{A},\quad B=\wt{B} \quad\mbox{and}\quad q = \wt{q}, \quad \mbox{in }\Omega.
	\]
\end{theorem}

\subsection*{Dirichlet boundary data}
Let us now consider a different boundary information, for the same problem given in \eqref{problem}. 
Consider \eqref{problem} given with the Dirichlet boundary conditions instead of Navier boundary conditions.
Let us denote 
\[ \gamma_Du =  \Big{(} u|_{\partial\Omega}, \partial_\nu u|_{\partial\Omega} \Big{)}.\]
Then for $f= (f_0,f_1) \in H^{\frac{7}{2}}(\partial\Omega)\times H^{\frac{5}{2}}(\partial\Omega)$ we consider the boundary value problem 
\begin{equation}\begin{aligned}\label{dirichlet}
\mathcal{L}(x,D)u &= 0 \quad\mbox{ in }\Omega, \\
\gamma_Du &= f \quad\mbox{ on }\partial\Omega. 
\end{aligned}
\end{equation}
The corresponding Neumann trace is 
\[\gamma_D^{\#} = \Big{(}\partial_{\nu}^{2}u|_{\partial\Omega}, \partial_{\nu}^{3}u|_{\partial\Omega}\Big{)} \in H^{\frac{3}{2}}(\partial\Omega)\times H^{\frac{1}{2}}(\partial\Omega),\]
where $u\in H^{4}(\Omega)$ is the unique solution to the Dirichlet problem \eqref{dirichlet}. See \cite{AGM,Gerd_Grubb} for the wellposedness of the forward problem \eqref{dirichlet}.
We introduce the set of Cauchy data for the operator $\mathcal{L}(x,D)$ with the Dirichlet 
boundary condition by
\begin{equation*}
\mathcal{C}^D = \{ \Big{(}u|_{\partial\Omega},\partial_{\nu}u|_{\partial\Omega}\ ;\ \partial^{2}_{\nu}u|_{\partial\Omega},\partial^{3}_{\nu}u|_{\partial\Omega}\Big{)} \st \Lc(x,D)u=0 \mbox{ in }\Omega \}.
\end{equation*}
As a corollary of Theorem \ref{mainresult}, we have the following result: 
\begin{corollary}
We assume $A,\widetilde{A}$, $B,\wt{B}$ and $q,\widetilde{q}$
satisfy the same conditions as in Theorem \ref{mainresult}. 
Let $\wt{\mathcal{C}^{D}}$ be the Cauchy data corresponding to the operator $\wt{\Lc}(x,D)$.
Then $\mathcal{C}^D= \widetilde{\mathcal{C}}^D$ on $\partial \O$ implies that $A = {\wt{A}}$, $B=\wt{B}$ and 
$q = \wt{q}$ in $\Omega$.
\end{corollary} 
\bpr Proceeding in a similar way as of Theorem \ref{mainresult}, in this case we end up with an integral identity same as \eqref{integralidentityE}. 
Then following the same analysis, one can show uniqueness of the lower order perturbations in $\Omega$.\epr

	\subsection*{A brief discussion on the techniques}
	A general approach to solve a Calder\'on type inverse problem follows from the pioneering work of \cite{SU}. By a clever use of integrations by parts formula and the equality of the DN map at the boundary, we obtain integral identities concerning the perturbation coefficients along with solutions of the operator and its adjoint under consideration. For instance, we obtain
	\begin{equation}\label{eq_1}
	\sum_{j,k=1}^{n}\int_{{\Omega}}\lb (A_{jk}-{\wt{A}_{jk}})D^{j}D^{k}\widetilde{u} + (B_{j}-{\wt{B}_{j}})D^{j}\widetilde{u} + (q-\widetilde{q})\widetilde{u}\rb\overline{v}\,\D x
	= 0\end{equation}
	where $\wt{u}$ solves $\Lc_{\wt{A},\wt{B},\wt{q}}(x,D)\wt{u}=0$ and $v$ solves $\Lc^{*}_{A,B,q}(x,D)v=0$ in $\Omega$. 
	
	Next we seek for particular class of solutions for $\wt{u}$ and $v$ which are known as the complex geometric optics (CGO) solutions for the operator $\Lc(x,D)$ and its adjoint $\Lc^{*}(x,D)$ respectively. They consist of a complex phase function and a complex amplitude. The amplitude can be expanded asymptotically as solutions of a series of transport equations. For the summability, we prove a Carleman estimate to bound the tail of the series with desired smallness, see \cite{KSU} for Schr\"{o}dinger operator. In our case, we construct them in Section \ref{Carleman Estimates}.
	
	Note that, the key to make this method work is to produce enough amplitudes, solving the transport equations and thus enough CGO solutions for the operator.
	The main challenge to recover an anisotropic second order perturbation of a biharmonic operator is to construct enough CGO solutions.
	The transport equations we get for the amplitudes are of the form of a second order partial differential operator with a potential term, governed by the unknown anisotropic coefficient $A$:
	\begin{equation}\label{eq_2}
	4((\mu_1+i\mu_2)\cdot\nabla)^2a +A(x)(\mu_1+i\mu_2)\cdot(\mu_1+i\mu_2)a=0,
	\end{equation}
	where $\mu_1,\mu_2\in\mathbb{R}^n$ satisfying $|\mu_1|=|\mu_2|$ and $\mu_1\perp\mu_2$. Note that, when $A$ is an anisotropic matrix then  $A(x)(\mu_1+i\mu_2)\cdot(\mu_1+i\mu_2)$ remains non-zero, however when $A$ is isotropic it remains always zero.
	 In order to provide enough $\wt{u},v$ in \eqref{eq_1}, we need to look for a rich class of solutions $a$ of \eqref{eq_2}.    
	   We construct yet another CGO type solution for the amplitudes $a$ and provide a sufficiently large class solutions. 
	So altogether, we construct CGO solutions, for the biharmonic operator under considerations, having CGO amplitudes.
	To the best of the author's knowledge, construction of \emph{CGO amplitudes} is new in the analysis of the Calder\'on problem. For this construction, we use a Carleman estimate based on the two dimensional $\overline{\partial}$-bar operator, which we prove in the due course.  We dedicate Section \ref{Sec_amplitude} to construct such \emph{CGO amplitudes} and provide a detailed discussion there.
		Finally, in Section \ref{determination}, using the large class of solutions constructed in the previous sections we show that $A=\wt{A}$, $B=\wt{B}$ and $q=\wt{q}$ in $\Omega$.
		\subsection*{Acknowledgement}
		The research of T.G. is supported by the Collaborative Research Center, membership no. 1283, Universit\"{a}t Bielefeld. S.B. is partly supported by Project no.: 16305018 of the Hong Kong Research Grant Council.
\section{Carleman estimate and CGO solutions}\label{Carleman Estimates}
In this section we construct complex geometric optics (CGO) type solutions of $\Lc(x,D)$ in \eqref{operator}. 
To construct complex geometric optics solutions we need certain   solvability result for the correction term, with desired decay estimates. We use the method of Carleman estimates to derive suitable weighted estimates for the operator $\Lc(x,D)$ and its formal $L^{2}(\Omega)$ adjoint $\Lc^{*}(x,D)$. Our method is essentially based on the Carleman estimates derived for the conjugated Laplacian operator with a gain of two derivatives \cite{ST}. 

\subsection{Interior Carleman estimates} 
Let us recall $\Lc(x,D)$ as in \eqref{operator}.
Note that $\Lc^{*}(x,D)$, the $L^2(\Omega)$ adjoint of $\Lc(x,D)$, has a similar form as of $\Lc(x,D)$ with possibly different coefficients $A^{\sharp}$, $B^{\sharp}$ and $q^{\sharp}$ (see \eqref{adjoint-operator}).
In this section, we prove an interior Carleman estimate for the conjugated semiclassical version of the operator $\Lc(x,D)$ as well as its adjoint operator. 

First we prove a Carleman estimate for the principal part of the semiclassical version of the operator $\Lc(x,D)$, which is given as $(-h^2\Delta)^2$.
Then by adding lower order terms to it finally we derive the required Carleman estimate for the conjugated semiclassical version of the operator $\Lc(x,D)$.
We start by recalling the definition of a limiting Carleman weight for the semiclassical Laplacian $(-h^2\Delta)$. 
Let $\wt{\O}$ be an open set in $\mathbb{R}^n$ such that 
$\Omega\ssubset \wt{\O}$ and let $\vp\in C^{\infty}(\wt{\O},\mathbb{R})$.
Consider the conjugated, semiclassical Laplacian operator
$P_{0,\vp} = e^{\frac{\vp}{h}}(-h^2\Delta)e^{-\frac{\vp}{h}}$ 
with its semiclassical symbol $p_{0,\vp}(x,\xi)=|\xi|^2-|\nabla_x\vp|^2+2i\xi\cdot\nabla_x\vp$. 

\begin{definition}[\cite{KSU}]
We say that $\vp$ is a limiting Carleman weight for $(-h^2\Delta)$
in $\wt{\O}$ if $\nabla\vp \neq 0 $ in $\wt{\O}$ and the Poisson 
bracket of $\mathrm{\mathrm{Re}}(p_{0,\vp})$ and  $\mathrm{Im}(p_{0,\vp})$ satisfies
\[\Big{\{}\mathrm{Re}(p_{0,\vp}), \mathrm{Im}(p_{0,\vp})\Big{\}}(x,\xi)=0, \quad \mbox{ whenever } p_{0,\vp}(x,\xi)=0, \quad \mbox{for } (x,\xi)\in (\overline{\wt{\O}}\times\mathbb{R}^n).
\]
\end{definition}
Examples of such $\vp$ are the linear weights defined as 
$\vp(x) = \A\cdot x$, where $\A\in\mathbb{R}^n\setminus\{0\}$ or 
the logarithmic weights $\vp(x)= \log|x-x_0|$ with $x_0 \notin \overline{\wt{\O}}$.
Throughout this article we consider the limiting Carleman weight to be of the form $\vp(x)=(\A\cdot x)$ where $\A\in\mathbb{R}^n$ with $|\A|=1$.

As the principal symbol of the semiclassical conjugated biharmonic operator $e^{\frac{\vp}{h}}(-h^2\Delta)^{2}e^{-\frac{\vp}{h}}$ is given by $p_{0,\vp}^2$, which is not of principal type, the idea of Carleman weights for biharmonic operators does not make sense. Instead we work with the limiting Carleman weights for the conjugated semiclassical Laplacian operator.
In order to get the Carleman estimate for the biharmonic operators we iterate the Carleman estimate obtained for the semiclassical Laplacian.

We use the semiclassical Sobolev spaces $H^s_{\mathrm{scl}}(\mathbb{R}^n)$ with $s\in\mathbb{R}$, equipped with the norm 
\begin{align*}
\lVert u\rVert_{H^s_{\mathrm{scl}}(\mathbb{R}^n)} = \lVert {\langle hD\rangle}^su\rVert_{L^2(\mathbb{R}^n)},
\end{align*} where 
$\langle \xi \rangle = (1+|\xi|^2)^{\frac{1}{2}}$. For a bounded domain $\Omega\subset \Rb^n$ with Lipschitz boundary, we define the semiclassical Sobolev space $H^s_{scl}(\Omega)$ as the restriction of $H^s_{scl}(\Rb^n)$ in $\Omega$ with the norm as 
\[	||u||_{H^s_{scl}(\Omega)} := \inf_{\substack{v \in H^s_{scl}(\mathbb{R}^n), \\ v|_{\Omega}=u}} \;||v||_{H^s_{scl}(\mathbb{R}^n)}.\]
%
For $s=m$ a positive integer we get
\[	\lVert u \rVert^2_{H^m_{scl}(\Omega)} \simeq \sum_{|\A|\leq m} \lVert (hD)^\A v \rVert^2_{L^2(\Omega)},
\]
where $\simeq$ denotes equivalence in the two norms on both sides of the above relation.
We define $H^s_{scl,0}(\O)$ to be closure of the $C^{\infty}_0(\O)$ in $H^s_{scl}(\O)$ for any $s>0$. For $s<0$ one can realize $H^{s}_{scl}(\O)$ to be the dual of the space $H^{-s}_{scl,0}(\O)$.

With these notations we now prove the following proposition.
\begin{proposition}\label{Prop: Interior Carleman Estimate}
	Let $A \in W^{2,\infty}(\Omega,\mathbb{C}^{n^{2}})$, $B \in W^{1,\infty}(\Omega,\mathbb{C}^{n})$, $q\in L^{\infty}(\O,\Cb)$
	and $\vp$ be a limiting Carleman weight for the semiclassical Laplacian on $\wt{\O}$. Then for $0< h\ll 1 $ and $-4\leq s\leq 0$, we have
	\begin{equation}\label{carlemanestimate3}
		h^{2}\lVert u\rVert_{H^{s+4}_{\mathrm{scl}}} \leq C \lVert h^4e^{\frac{\vp}{h}}\Lc(x,D)e^{-\frac{\vp}{h}}u\rVert_{H^{s}_{\mathrm{scl}}},\qquad \mbox{for all } u\in C^{\infty}_0(\Omega).
	\end{equation}
	the constant $C = C_{s,\O, A,B,q}$ is independent of $h>0$. 
\end{proposition}
\begin{proof}
Let us consider the convexified Carleman weight (see \cite{KSU}) defined as 
		\[
		\vp_{\ve}=\vp+ \frac{h}{2\ve} \vp^{2} \mbox{ on } \wt{\O}.
		\]
We begin with the Carleman estimate for the semiclassical Laplacian with a gain of two derivatives proved in \cite{ST}:
	\begin{equation}\label{carlemanesti}
	\frac{h}{\sqrt{\epsilon}}\lVert u\rVert_{H^{s+2}_{\mathrm{scl}}} \leq C \LM e^{\frac{\vp_{\ve}}{h}}(-h^2\Delta) e^{-\frac{\vp_{\ve}}{h}}u\RM_{H^s_{\mathrm{scl}}}, \quad \mbox{for all } u\in C^{\infty}_0(\Omega), \quad s \in \Rb.
	\end{equation}
Let $-4\leq s\leq 0$, then iterating the estimate \eqref{carlemanesti} for $2$ times, we get the following estimate: 
\begin{equation}\label{carlemanestimate22}
	\left(\frac{h}{\sqrt{\epsilon}}\right)^2\lVert u\rVert_{H^{s+4}_{\mathrm{scl}}} \leq C \LM e^{\frac{\vp_{\ve}}{h}}(-h^2\Delta)^2 e^{-\frac{\vp_{\ve}}{h}}u\RM_{H^s_{\mathrm{scl}}}, \quad \mbox{for all } u\in C^{\infty}_0(\Omega).
\end{equation}

Next to get the required Carleman estimate for $\mathcal{L}(x,D)$, we add the lower order perturbations (given in \eqref{operator}) to the above estimate. 
We first take the zero-th order term $(h^{4}q)$, where $q\in L^{\infty}(\Omega,\mathbb{C})$, and we get 
\[
\lVert h^{4} qu\rVert _{H^{s}_{\mathrm{scl}}} \leq h^{4}\lVert q\rVert_{L^{\infty}} \lVert u\rVert_{L^2} \leq h^{4} \lVert q\rVert_{L^{\infty}}\lVert u\rVert_{H^{s+4}_{\mathrm{scl}}}, \quad -4\leq s\leq 0.
\]
Next we consider the first order term $h^{4}(B\cdot D)$, where $B \in W^{1,\infty}(\Omega,\Cb^n)$. 
We observe
\begin{equation}\label{First order term}
h^{3}e^{\frac{\vp_{\ve}}{h}}\sum\limits_{j=1}^{n} B_{j}(hD)^{j} e^{-\frac{\vp_{\ve}}{h}}u
= h^{3}\sum\limits_{j=1}^{n} B_{j}(-D^{j}\vp_{\ve}+ hD^{j})u.
\end{equation}
The first term in the right hand side of \eqref{First order term} can be estimated as 
\[	\lVert(B\cdot D\vp_{\epsilon})u\rVert_{H^{s}_{\mathrm{scl}}} \leq \lVert B\cdot D\vp_{\epsilon}\rVert_{L^{\infty}}\lVert u\rVert_{H^{s+4}_{\mathrm{scl}}},\quad -4\leq s\leq 0.
\]
Now as $\vp_{\epsilon} = \vp + \frac{h}{\epsilon}\vp^2 $ and $0 < h \ll \epsilon \ll 1$, that is, $0 < \frac{h}{\epsilon} < 1 $. Hence $ \lVert D^{\A} \vp_{\epsilon}\rVert_{L^{\infty}} = \Oc(1)$ for any $\A$ and consequently we get
\[ \lVert(B\cdot D\vp_{\epsilon})u\rVert_{H^{s}_{\mathrm{scl}}} \leq \mathcal{O}(1)\lVert u\rVert_{H^{s+4}_{\mathrm{scl}}}.\]
For the second term in the right hand side of \eqref{First order term} we observe that for $-4\leq s\leq 0$ we have
\begin{equation*}\begin{aligned}
\lVert B \cdot (hD)u\rVert_{H^{s}_{\mathrm{scl}}} &\leq \lVert hD \cdot(Bu)\rVert_{H^{s}_{\mathrm{scl}}} + h \lVert (D \cdot B)u\rVert_{H^{s}_{\mathrm{scl}}}\\
&\leq \Oc(1)\lVert Bu\rVert_{H^{s+1}_{\mathrm{scl}}} + \Oc(h)\lVert u\rVert_{H^{s+4}_{\mathrm{scl}}}\\ 
&\leq \Oc(1)\lVert u\rVert_{H^{s+4}_{\mathrm{scl}}}.
\end{aligned}\end{equation*}
The last inequality follows from the fact that the operator given as multiplication by $B$ is continuous from $H^{s+4}_{\mathrm{scl}}$ to $H^{s+1}_{\mathrm{scl}}$ where $B\in W^{1,\infty}$. To prove the continuity, it suffices to consider the complex interpolation for the cases $s = 0$ and $s = -4$.
Hence, we have 
\begin{equation}\label{newline1}
\lVert h^{3}e^{\frac{\vp_{\epsilon}}{h}}(B\cdot hD) e^{-\frac{\vp_{\epsilon}}{h}}u\rVert_{H^s_{\mathrm{scl}}} \leq \Oc(h^3)\lVert u\rVert_{H^{s+4}_{\mathrm{scl}}}, \quad -4\leq s\leq 0.   
\end{equation}
	
Now consider the term $h^{4}A_{\A}D^{\A}e^{\varphi_{\epsilon}/h}u$, for $\lvert \A \rvert = 2$, where $A$ is a symmetric matrix. We have
\begin{equation}\label{2nd order term}
\begin{aligned}
&h^{2}e^{\frac{\vp_{\ve}}{h}}\sum\limits_{\lvert\A\rvert=2}A_{\A}h^2D^{\A} e^{-\frac{\vp_{\ve}}{h}}u
= h^{2} \sum_{j,k=1}^{n} A_{jk}(D^{j}\vp_{\ve}D^{k}\vp_{\ve}-hD^{j}D^{k}\vp_{\ve} + 2hD^{j}\vp_{\ve}D^{k} + h^2 D^{j}D^{k})u.
\end{aligned}
\end{equation}
For the first two terms in the right hand side of \eqref{2nd order term}, we get
\[\begin{aligned}
\lVert A_{jk} \lb D^{j}\vp_{\ve}D^{k}\vp_{\ve} -hD^{j}D^{k}\vp_{\ve}\rb u \rVert_{H^s_{\mathrm{scl}}}
&\leq C\lVert A_{jk} \lb D^{j}\vp_{\ve}D^{k}\vp_{\ve} -hD^{j}D^{k}\vp_{\ve}\rb \rVert_{L^{\infty}}\lVert u\rVert_{H^s_{\mathrm{scl}}}\\
&\leq \mathcal{O}(1)\lVert u\rVert_{H^{s+4}_{\mathrm{scl}}}, \qquad -4\leq s\leq 0. 
\end{aligned}\]
Analysing the third term in \eqref{2nd order term}, we see
\[\begin{aligned}
\lVert A_{jk} D^{j}\vp_{\ve}hD^{k}u \rVert_{H^s_{\mathrm{scl}}}
\leq& C\lVert hD^{k}\left(A_{jk} u D^{j}\vp_{\ve}\right) \rVert_{H^{s}_{\mathrm{scl}}} 
	+ Ch\lVert D^{k}\left(A_{jk}D^{j}\vp_{\ve}\right) \rVert_{L^{\infty}}\lVert u\rVert_{H^{s+4}_{\mathrm{scl}}} \\
\leq& \mathcal{O}(1)\lVert \left(A_{jk} D^{j}\vp_{\ve}\right) u\rVert_{H^{s+1}_{\mathrm{scl}}}
	+ Ch\lVert u\rVert_{H^{s+4}_{\mathrm{scl}}}\\
\leq& \mathcal{O}(1)\lVert u\rVert_{H^{s+4}_{\textrm{scl}}}, \quad -4\leq s\leq 0,
\end{aligned}\]
where for the first term we use the continuity of the  multiplication operator $A_{jk} : H^{s+4}_{\mathrm{scl}}\to H^{s+1}_{\mathrm{scl}}$ whenever $A_{jk}\in W^{2,\infty}$.

Now, consider the last term of the expression on the right hand side of \eqref{2nd order term}, we get 
\[\begin{aligned}
\lVert A_{jk}h^2D^{j}D^{k}u \rVert_{H^s_{\mathrm{scl}}} 
&\leq\lVert h^2D^{j}D^{k}(A_{jk} u) \rVert_{H^s_{\mathrm{scl}}} 
+ 2\lVert h^2D^{j}(A_{jk}) D^{k}u \rVert_{H^s_{\mathrm{scl}}} 
+\lVert h^2\left(D^{j}D^{k}(A_{jk})\right) u \rVert_{H^{s}_{\mathrm{scl}}}\\
&\leq \mathcal{O}(1)\lVert A_{jk} u \rVert_{H^{s+2}_{\mathrm{scl}}} + \mathcal{O}(h)\lVert u \rVert_{H^{s+4}_{\mathrm{scl}}} + \mathcal{O}(h^2)\lVert u \rVert_{H^{s+4}_{\mathrm{scl}}}\\
&\leq \mathcal{O}(1)\lVert  u \rVert_{H^{s+4}_{\mathrm{scl}}} + \mathcal{O}(h)\lVert u \rVert_{H^{s+4}_{\mathrm{scl}}} + \mathcal{O}(h^2)\lVert u \rVert_{H^{s+4}_{\mathrm{scl}}}, \quad -4\leq s\leq 0.
\end{aligned}\]
Here in the first term we use the continuity of the multiplication operator $A_{jk} : H^{s+4}_{\mathrm{scl}}\to H^{s+2}_{\mathrm{scl}}$ whenever $A_{jk}\in W^{2,\infty}$. 
The inequality on the second term on the right hand side follows from using the continuity of the  multiplication operator $A_{jk} : H^{s+4}_{\mathrm{scl}}\to H^{s+1}_{\mathrm{scl}}$ whenever $A_{jk}\in W^{2,\infty}$.  

Adding all the lower order terms in \eqref{carlemanestimate22}, choosing $h\ll \epsilon \ll 1$ small enough and using the standard bounds i.e. 
$1\leq e^{\frac{\vp^2}{2\epsilon}}\leq C, \hspace{3pt} \frac{1}{2}\leq  1 + \frac{h}{\epsilon}\vp \leq \frac{3}{2}$, 
we finally get our desired estimate \eqref{carlemanestimate3}.
\end{proof}

Let us denote 
\[
\Lc_{\vp}(x,D) = h^{4}e^{\frac{\vp}{h}}\Lc(x,D)e^{-\frac{\vp}{h}}.
\]
The formal $L^{2}$ adjoint of $\Lc_{\vp}(x,D)$ would be ${\Lc}^{*}_{\vp}(x,D)=h^{4}e^{-\frac{\vp}{h}}\Lc^{*}(x,D)e^{\frac{\vp}{h}}$, where $\Lc^{*}(x,D)$ is the formal $L^2$-adjoint of the operator $\Lc(x,D)$. 

As $\Lc^{*}(x,D)$ has the similar form as $\Lc(x,D)$ and $-\vp$ is a limiting Carleman weight if $\vp$ is, the Carleman estimate derived in Proposition \ref{Prop: Interior Carleman Estimate} also holds for ${\Lc}^{*}_{\vp}(x,D)$ as well.
From \eqref{adjoint-operator} we see that $\Lc_{A,B,q}^*(x,D) = \Lc_{A^{\sharp},B^{\sharp},q^{\sharp}}(x,D)$. Since we consider $A \in W^{3,\infty}(\O)$ and $B \in W^{2,\infty}(\O)$, hence, $A^{\sharp}\in W^{3,\infty}(\O)$, $B^{\sharp} \in W^{2,\infty}(\O)$. That is, we have the following result regarding the Carleman estimate for the adjoint operator.
\begin{corollary}\label{Cor: Interior Carleman Estimate}
	Let $A \in W^{2,\infty}(\Omega,\mathbb{C}^{n^{2}})$, $B \in W^{1,\infty}(\Omega,\mathbb{C}^{n})$, $q\in L^{\infty}(\O,\Cb)$
	and $\vp$ be a limiting Carleman weight for the semiclassical Laplacian on $\wt{\O}$. Then for $0< h\ll 1 $ and $-4\leq s\leq 0$, we have
	\begin{equation}\label{carlemanestimate4}
		h^{2}\lVert u\rVert_{H^{s+4}_{\mathrm{scl}}} \leq C \lVert h^{4}e^{\pm\frac{\vp}{h}}\Lc^*(x,D)e^{\mp\frac{\vp}{h}}u\rVert_{H^{s}_{\mathrm{scl}}},\qquad \mbox{for all } u\in C^{\infty}_0(\Omega).
	\end{equation}
	the constant $C = C_{s,\O, A,B,q}$ is independent of $h>0$. 
\end{corollary}

Let us now convert the Carleman estimate \eqref{carlemanestimate3} for $\mathcal{L}_{\varphi}^{*}$ into a solvability result for $\mathcal{L}_{\varphi}$.
\begin{proposition}\label{existence}
	Let $A\in W^{2,\infty}(\Omega,\mathbb{C}^{n^{2}})$, $B \in W^{1,\infty}(\Omega,\mathbb{C}^{n^{2}})$ and  $q\in L^{\infty}(\O,\Cb)$ and $\vp$ be any limiting Carleman weight for the semiclassical Laplacian on $\wt{\O}$. 
	For $0< h\ll 1$ sufficiently small, the equation 
	\begin{equation}\label{huv}
	\Lc_{\vp}(x,D)u = v \mbox{ in }\Omega,
	\end{equation}
	has a solution $u\in H^4(\Omega)$, for $v\in L^2(\Omega)$ 
	satisfying, 
	\begin{equation}\label{hdecay}
	{h}^2\lVert u(\cdot;h) \rVert_{H^{2}_{\mathrm{scl}}(\Omega)} \leq C \lVert v \rVert_{L^{2}(\Omega)}.
	\end{equation}
	The constant $C>0$ is independent of $h$ and depends only on $A$, $B$ and $q$.
\end{proposition}
The proof of the above result follows from a standard functional analysis argument (see \cite{DOS}).
\subsection{Construction of C.G.O. solutions}\label{ccs}
Here we construct complex geometric optics type solutions of the equation $\Lc(x,D)u=0$ and its $L^2(\Omega)$ conjugate, based on Proposition \ref{existence}. 
We propose a solution of $\Lc(x,D)u=0$ in the form
\begin{equation}\label{CGO} 
u = e^{\frac{(\vp + \I\psi)}{h}}(a_0(x)+ha_1(x) + r(x;h)),
\end{equation}
where $0<h\ll 1$, $\vp(x)$ is a limiting Carleman weight for the semiclassical Laplacian. 
The real valued phase function $\psi$ is chosen such that $\psi$ is smooth near $\overline{\Omega}$ and solves the Eikonal equation $p_{0,\vp}(x,\nabla\psi) =0$ in $\wt{\O}$. 
The functions $a_0$ and $a_1$ are the complex amplitudes solving certain transport equations. 
The function $r(x;h)$ is the correction term which satisfies the following estimate $\lVert r\rVert_{H^{4}_{\mathrm{scl}}} = \Oc(h^2)$.

We consider $\varphi$ and $\psi$ to be 
\begin{equation}\label{phi and psi}
\vp(x)=\mu_1 \cdot x,\quad 
\psi(x)= \mu_2 \cdot x,
\end{equation}
where $\mu_1,\mu_2 \in \Rb^n\setminus\{0\}$ are such that $\mu_1\cdot\mu_2 = 0$ and $\lvert \mu_2 \rvert = \lvert \mu_1 \rvert$. Observe that $\varphi$ and $\psi$ solves the eikonal equation $p_{0,\vp}(x,\nabla\psi) =0$ in $\wt{\O}$, that is $\Lm\nabla \vp\Rm = \Lm\nabla \psi\Rm$ and $\nabla\vp\cdot\nabla\psi =0$.
 
In order to obtain the transport equations for the amplitudes $a_0$ and $a_1$, we define the transport operator \begin{equation}\label{T}T_{\varphi, \psi}= \left[(\nabla\vp +\I\nabla\psi)\cdot\nabla\right]\end{equation} and expand the conjugated operator as
\begin{equation}\label{ConjugatedOperator}
\begin{aligned}
e^{-\frac{(\vp+ \I\psi)}{h}}h^{4}\Lc(x,D)e^{\frac{(\vp + \I\psi)}{h}}
=& (-h^2\Delta -2hT)^2\\
&+ \sum_{j,k=1}^{n} h^{2} A_{jk}\left( D_{j}(\varphi+i\psi)D_{k}(\varphi+i\psi) + 2h D_{j}(\varphi+i\psi)D^{k} + h^2D_{j}D_{k} \right)\\
&+ \sum_{j=1}^{n} h^{3} B_{j}\left( D_{j}(\varphi + i\psi) + hD_{j}\right) + h^{4}q.
\end{aligned}
\end{equation}
We solve for $a_0$ and $a_1$ satisfying
\begin{equation}\label{transportequation1}
(-2T)^2 a_0 +  \sum_{j,k=1}^{n} A_{jk} D^{j}(\varphi+i\psi)\cdot D^{k}(\varphi+i\psi)\, a_0 = 0,
\end{equation}
\begin{equation}\label{transportequation2}
(-2T)^2 a_1 + \sum_{j,k=1}^{n} A_{jk} D^{j}(\varphi+i\psi)\cdot D^{k}(\varphi+i\psi)\, a_1 
= -2(T\circ \Delta + \Delta\circ T)a_0 - \left(B\cdot D(\vp+\I\psi)\right)a_0
\end{equation}
We will present a detailed proof for existence of complex amplitudes $a_0\in H^4(\Omega)$, $a_1\in H^4(\Omega)$ satisfying \eqref{transportequation1} and \eqref{transportequation2} in Section \ref{Sec_amplitude}.
For the time being let us assume that such $a_0$ and $a_1$ exists in $H^4(\O)$. 
Having chosen the amplitudes $a_0$, $a_1$ in this way, from \eqref{ConjugatedOperator} we obtain
\begin{equation}\label{r}
\begin{aligned}
e^{-\frac{(\vp+ i\psi)}{h}}h^{4}\Lc(x,D)\left(e^{\frac{(\vp + i\psi)}{h}}r(x,h)\right)
=&-h^{4}\Lc(x,D)\left(a_0 + ha_1\right)\\ 
&-2h^4(T\circ \Delta + \Delta\circ T)a_1 -h^4\left(B\cdot D(\vp+\I\psi)\right)a_1
\end{aligned}
\end{equation}
Thanks to Proposition \ref{existence}, for $h > 0$ small enough, there exists a solution $r \in H^4(\Omega)$ of \eqref{r} with the decay estimate 
\begin{equation}
\|r\|_{H^4_{scl}}= \mathcal{O}(h^2).
\end{equation}
Summing up, we have the following result.
\begin{proposition}\label{solvibility}
Let us consider the equation
\begin{equation}\label{L-sharp}
\Lc(x,D)u =  (-\Delta)^2 u + \sum_{j,k=1}^{n} {A}_{jk} D^{j}D^{k}u + \sum_{j=1}^{n} B_{j} D^{j} u + qu =0,
\end{equation}
where $A\in W^{3,\infty}(\Omega,\mathbb{C}^{n^{2}})$, $B\in W^{2,\infty}(\Omega,\mathbb{C}^{n})$ and $q\in L^{\infty}(\Omega,\mathbb{C})$.
Then for all $0< h \ll 1$, there exists a solution $u\in H^4(\O)$ of \eqref{L-sharp} of the form
\begin{equation}\label{qw10} 
u(x,h) = e^{\frac{\vp(x) + i\psi(x)}{h}}(a_0(x) +ha_1(x) + r(x;h)) 
\end{equation}
where $\vp$ and $\psi$ are as in \eqref{phi and psi} real valued linear harmonic functions conjugate to each other. Here $a_0,a_1 \in H^4(\Omega)$ are complex amplitudes satisfying the transport equations \eqref{transportequation1}, \eqref{transportequation2} and $r\in H^4(\Omega)$ satisfies the estimate $\lVert r\rVert_{H^{4}_{\mathrm{scl}}} = \Oc(h^2)$.
\end{proposition}
\begin{remark}
Note that in the above proposition, we assumed an extra regularity of $A$ and $B$ to be in $W^{3,\infty}$ and $W^{2,\infty}$ respectively. The reason being, our proof of the existence of the complex amplitudes $a_0$ and $a_1$ in $H^4(\O)$, satisfying the transport equations \eqref{transportequation1} and \eqref{transportequation2} respectively, requires that; See also Remark \ref{reg2}.
\end{remark}
\subsection*{The adjoint operator}
Let us now calculate the formal $L^{2}(\O)$ adjoint of the operator $\Lc(x,D)$ as
\begin{equation}\begin{aligned}\label{adjoint-operator}
		\Lc^{*}(x,D):= (-\Delta)^2 + \sum_{j,k=1}^{n}
		A^{\sharp}_{jk}(x)D^{j}D^{k} + \sum_{j=1}^{n}
		B^{\sharp}_{j}(x)D^{j} + q^{\sharp}(x),
\end{aligned}\end{equation}
where
\[\begin{cases}
	A^{\sharp}_{jk}(x) &= \overline{A_{jk}(x)}, \quad \mbox{for } j,k = 1,2,\dots,n,\\
	B^{\sharp}_{k}(x)  &= \overline{B_k(x)} + \sum_{j=1}^{n} D^{j}\overline{A}_{jk}(x), \quad \mbox{for } k = 1,\dots,n,\\
	q^{\sharp}(x) &= \overline{q(x)} + \sum_{j,k=1}^{n} D^{j}D^{k}\overline{A}_{jk}(x) + \sum_{j=1}^{n} D^{j}\overline{B}_{j}(x).
\end{cases}\]
Since the form of $\Lc^*(x,D)$ is same as that of $\Lc(x,D)$ with possibly different coefficient, we may argue that we can construct a CGO solution for $\Lc^*(x,D)v=0$ in $\O$.
We state it formally in the following remark.
\begin{remark}\label{Rem_adjoint_op}
There exist solution $v \in H^4(\O)$, solving $\Lc^*(x,D)v=0$ in $\O$, of the form
\[	v(x,h) = e^{\frac{-\vp(x) + i\psi(x)}{h}}(a_0^{\sharp}(x) +ha_1^{\sharp}(x) + r^{\sharp}(x;h)), 
\]
provided we have $a_0^{\sharp},a_1^{\sharp} \in H^4(\O)$ satisfying
\[
\begin{aligned}
&(-2T_{-\vp,\psi})^2 a^{\sharp}_0 +  \sum_{j,k=1}^{n} A^{\sharp}_{jk} D^{j}(-\varphi+i\psi)\cdot D^{k}(-\varphi+i\psi)\, a^{\sharp}_0 = 0\quad\mbox{ in }\Omega\\[4pt]
&(-2T_{-\vp,\psi})^2 a^{\sharp}_1 + \sum_{j,k=1}^{n} A^{\sharp}_{jk} D^{j}(-\varphi+i\psi)\cdot D^{k}(-\varphi+i\psi)\, a^{\sharp}_1 \\
&\qquad\qquad\qquad=-2(T_{-\vp,\psi}\circ \Delta + \Delta\circ T_{-\vp,\psi})a^{\sharp}_0 - \left(B^{\sharp}\cdot D(-\vp+\I\psi)\right)a^{\sharp}_0\quad\mbox{in }\Omega.
\end{aligned}
\]
\end{remark}

Now for existence of the amplitudes required in Proposition \ref{solvibility} and Remark \ref{Rem_adjoint_op} we move into the next section, where we show that the amplitudes $a_0$ and $a_1$ exist with suitable regularity and moreover, one can construct CGO type forms of the amplitudes solving homogeneous higher order transport equations.

\section{Analysis on the amplitudes}\label{Sec_amplitude}
In this section we prove existence of the solutions of \eqref{transportequation1} and \eqref{transportequation2} having a specific form. This part is crucial in our analysis since we get a second order transport equation with a potential term for the amplitudes which does not appear in the previous works. In the earlier works on the Schr\"odinger and the magnetic Schr\"odinger operators \cite{SU,SUN,DOS,D_K_S_U,KS13,Chung,B18} we encounter only first order transport equations. On the other hand, the works on biharmonic and polyharmonic operators \cite{KRU2,KRU1,VT,BG19} we get potential free higher order transport equations for the amplitudes that $((\mu_1+i\mu_2)\cdot\nabla)^2a_0=0$, which can be dealt with in the same way as for the first order transport equations, see \cite{BG19}.

Let us first discuss the solvability of the transport equations as in \eqref{transportequation1}:
\begin{equation}\label{nte1}
4\big((\mu_1+i\mu_2)\cdot\nabla_x\big)^2 a_0 + \big( A(x)(\mu_1+i\mu_2)\cdot(\mu_1+i\mu_2)\big)\, a_0 = 0 \quad\mbox{in }\Omega.
\end{equation} 

\vspace{2pt}
\noindent
Recall that, here $\mu_1\perp\mu_2$ are unit vectors in $\mathbb{R}^n$. Let us invoke the change of variables $x \mapsto (t,s,x')$ such that
\[\mu_1\cdot\nabla = \partial_t
\mbox{ and }  \mu_2\cdot \nabla =\partial_s,\]
where $(t,s)\in \mathbb{R}^2$ and $x' \in \Rb^{n-2}$.
Let us define
\[	\Sigma_{x'} := \{(t,s) \in \Rb^2 : (t,s,x') \in \O\}.
\]
Consider the $\Rb^2$ plane and take $z=t+is \in \mathbb{C}$ so that
we can realize the operator $(\mu_1+i\mu_2)\cdot\nabla_x$ as $2\overline{\partial_z}$ operator. This says essentially \eqref{nte1} fall into the category of solving equations like 
\begin{equation}\label{me} (( e_1+ ie_2)\cdot \nabla_{t,s})^2a_0 + c(t,s)a_0=0 \quad\mbox{in }\Sigma\subset\mathbb{R}^2\end{equation}
where $\Sigma\subset \mathbb{R}^2$ is a regular bounded open set, $c(t,s)\in L^\infty(\Sigma;\Cb)$ and $e_1=(1,0),e_2=(0,1) \in \Rb^2$ are elements of the standard Euclidean basis vectors of $\Rb^2$.

\vspace{2pt}
In order to solve \eqref{me}, we propose complex geometric optics solutions as
\begin{equation}\label{ansatz}
a_0= e^{\frac{\widetilde{\varphi}-i\widetilde{\psi}}{\h}}(b_0+\rho(\cdot;\h))
\end{equation}
where $b_0\in C^\infty(\overline{\Sigma})$ is non-zero, $\widetilde{\vp}, \widetilde{\psi}$ are real valued linear harmonic functions,  conjugate to each other, and $ \rho$ to be determined with the desired decay estimate with respect to $\tau>0$ small enough.

We first decouple the operator $(\pm\partial_t + i\partial_s )^2$ into real and complex parts to have the following Carleman estimates. 
\begin{proposition}\label{CAR}
	Let $\Sigma\subset \mathbb{R}^2$ be a regular bounded open set. Let $\widetilde{\varphi}(t,s)= at+bs$ be a linear function, $(a,b)\neq (0,0)$. We have
	\begin{equation}\label{carR} \|Re\,\left(e^{-\frac{\widetilde{\varphi}}{\h}}\h^2(\pm\partial_t + i\partial_s )^2e^{\frac{\widetilde{\varphi}}{\h}}\right)w\|_{L^2} \geq \,C_1\h\,\|w\|_{L^2},
	\quad w\in C^\infty_0(\Sigma)\end{equation}
	and
	\begin{equation}\label{CAR_1} \left\|Im\,\left(e^{-\frac{\widetilde{\varphi}}{\h}}\h^2(\pm\partial_t + i\partial_s )^2e^{\frac{\widetilde{\varphi}}{\h}}\right)w\right\|_{L^2} \geq \,C_2\h\,\|w\|_{L^2},
	\quad w\in C^\infty_0(\Sigma)\end{equation}
	where $C_1,C_2>0$ are independent of $0<\h<1$,  $w$.
\end{proposition}
\bpr
Let us define
\begin{align*}
P_\pm  &:= Re\,e^{-\frac{\widetilde{\varphi}}{\h}}\h^2(\pm\partial_t + i\partial_s )^2e^{\frac{\widetilde{\varphi}}{\h}}\\[1mm] &=\big(\tau^2(D_2^2-D^2_1) + (a^2-b^2) \big)+\, i\,2\tau(\pm aD_1-bD_2)
\end{align*}
where $D_1=\frac{1}{i}\partial_t$ and $D_2=\frac{1}{i}\partial_s$. 

Let us write 
\[P_\pm =A_\pm+ iB_\pm\]
where $A^{*}_\pm = A_\pm$ and $B^{*}_\pm = B_\pm$ given as 
\[\begin{cases}
A_\pm = \frac{P_\pm +  P^{*}_\pm}{2}=\tau^2(D_2^2-D^2_1) + ( a^2-b^2) \\[2mm]
B_\pm = \frac{P_\pm- P^{*}_\pm}{2i}= 2\tau(\pm aD_1-bD_2).
\end{cases}\]
Now we have 
\[
\|P_\pm w\|^2_{L^2} = \|A_\pm w\|^2 + \|B_\pm w\|^2 + i([A_\pm,B_\pm]w|w),\quad w\in C^{\infty}_0(\Sigma)
\]
where $[A_\pm,B_\pm]=A_\pm B_\pm - B_\pm A_\pm$ is the commutator of $A_\pm$ and $B_\pm$. Since in our
case $A_\pm$ and $B_\pm$ are constant coefficient differential operators, thus $[A_\pm, B_\pm] = 0$.

Therefore,
\[
\|P_\pm w\|^2_{L^2} = \|A_\pm w\|^2_{L^2} + \|B_\pm w\|^2_{L^2} \geq \|B_\pm w\|^2_{L^2}.
\]
Now by using the Poincar\'{e} inequality\footnote{Let $\A \in \Rb^n$ be some non-zero vector and $S:=\{x \in \Rb^n \st k_1<\A\cdot x<k_2 \}$ be some unbounded strip for some $k_1,k_2 \in \Rb$ (note that, $S$ can contain any bounded set in $\mathbb{R}^n$).  Then one has the Poincar\'{e} inequality (c.f. \cite{SALO}):
	\[ \| (\A\cdot D)u\|_{L^2(S)} \geq C\|u\|_{L^2(S)},\quad \forall u \in C^{\infty}_c(S). 
	\]} we simply obtain 
\[
\|P_\pm w\|^2_{L^2}  \geq C_1\tau^2\,\|w\|^2_{L^2}, \quad w\in C^{\infty}_0(\Sigma)
\]
where $C_1>0$ is independent of $\tau>0$ and $w$.  

This completes the proof of \eqref{carR}. Similarly, one proves \eqref{CAR_1}. We give a quick sketch of that. 

Let 
\begin{align*}
Q_\pm  &:= Im\,e^{-\frac{\widetilde{\varphi}}{\h}}\h^2(\pm\partial_t + i\partial_s )^2e^{\frac{\widetilde{\varphi}}{\h}}\\[1mm] &= \underbrace{\pm2\big(-\tau^2D_{12}^2 + ab \big)}_{\wt{A}_\pm}+ \, i\,\underbrace{2\tau(bD_1 +a D_2)}_{\wt{B}_\pm}.
\end{align*}
Here $\wt{A}^{*}_\pm = \wt{A}_\pm$, $\wt{B}^{*}_\pm = \wt{B}_\pm$, and $\wt{A}_\pm$, $i\wt{B}_\pm$ are the self-adjoint and
skew-adjoint parts of $Q_\pm$ respectively, to have as before
\[
\|Q_\pm w\|^2_{L^2} = \|\wt{A}_\pm w\|^2_{L^2} + \|\wt{B}_\pm w\|^2_{L^2} \geq \|\wt{B}_\pm w\|^2_{L^2}\geq C_2\tau^2\,\|w\|^2_{L^2}, \quad w\in C^{\infty}_0(\Sigma)
\]
where the last inequality is due to Poincar\'{e}, and $C_2>0$ is independent of $\tau$ and $w$.
\epr

\begin{proposition}\label{scaling_2}
	Let $\Sigma$, $\widetilde{\varphi}$ are as in Proposition \ref{CAR}. Then for $\h>0$ small enough and $c=c_1+ic_2 \in L^{\infty}(\Sigma,\Cb)$ we get the following estimates
	\begin{equation}\label{P_3}
	\left\|Re\,\left(e^{-\frac{\widetilde{\varphi}}{\h}}\h^2(\pm\partial_t + i\partial_s )^2e^{\frac{\widetilde{\varphi}}{\h}}\right)w + \h^2c_1 w\right\|_{L^2} \geq \,C\h\,\|w\|_{L^2},
	\quad w\in C^\infty_0(\Sigma)
	\end{equation}
	and
	\begin{equation}\label{P_4}
	\left\|Im\,\left(e^{-\frac{\widetilde{\varphi}}{\h}}\h^2(\pm\partial_t + i\partial_s )^2e^{\frac{\widetilde{\varphi}}{\h}}\right)w + \h^2c_2w\right\|_{L^2} \geq \,C\h\,\|w\|_{L^2},
	\quad w\in C^\infty_0(\Sigma)
	\end{equation}
where $C>0$ is independent of $0<\tau<1$ and $w$.  
\end{proposition}
\bpr
Since $\| c_1(x) w(x)\|_{L^2} \leq C\| w\|_{L^2}$, thus  for $\h>0$ small enough 
\[	\begin{aligned}
&\left\|Re\,\left(e^{-\frac{\widetilde{\varphi}}{\h}}\h^2(\pm\partial_t + i\partial_s )^2e^{\frac{\widetilde{\varphi}}{\h}}\right)w + \h^2c_1 w\right\|_{L^2}\\
&\quad \geq \left\|Re\,\left(e^{-\frac{\widetilde{\varphi}}{\h}}\h^2(\pm\partial_t + i\partial_s )^2e^{\frac{\widetilde{\varphi}}{\h}}\right)w\right\|_{L^2} 
- \h^2\| c_1w \|_{L^2} \\
&\quad \geq \,C\h\,\|w\|_{L^2} - C\h^2\,\|w\|_{L^2} \geq C\h\,\|w\|_{L^2},
\qquad w\in C^\infty_0(\Sigma).
\end{aligned}\]
This proves \eqref{P_3}.
Similarly we prove \eqref{P_4}. 
\epr
Finally, as an application of the above results we prove the following existence result.

\begin{proposition}[Existence result]\label{ex_le}
	Let $\Sigma$, $\widetilde{\varphi}$ are as in Proposition \ref{CAR}, and $c\in L^{\infty}(\Sigma)$. Then for $\h>0$ small enough, and for a given $v\in L^2(\Sigma)$ the equation 
	\[ T_{\widetilde{\varphi}, \pm} u :=		e^{-\frac{\widetilde{\varphi}}{\h}}\,\tau^2\,((\pm e_1+ ie_2)\cdot \nabla)^2 e^{\frac{\widetilde{\varphi}}{\h}}u +\tau^2\,c(t,s)u =v
	\]
	has a solution $u\in L^2(\Sigma)$,
	satisfying,
	\begin{equation}\label{tau_esti}
	\h\lVert u \rVert_{L^2} \leq C \lVert v \rVert_{L^2}
	\end{equation}
	where the constant $C>0$ is independent of $\h>0$ and $u$. 
\end{proposition} 
\begin{proof}
	Let us observe that the formal adjoint of $T_{\widetilde{\varphi}, \pm}$ is 
	\[	T^*_{\widetilde{\varphi},\pm} = e^{\frac{\widetilde{\varphi}}{\h}}\,\tau^2\,((\mp e_1+ ie_2)\cdot \nabla)^2 e^{-\frac{\widetilde{\varphi}}{\h}}u +\h^2\,\overline{c},	\]
	so, Proposition \ref{scaling_2} is also true for $Re\, T^*_{\widetilde{\varphi},\pm}$ and $Im\,T^*_{\widetilde{\varphi},\pm}$ as well.
	
	Let us define \[D := \{(Re\,T^{*}_{\widetilde{\varphi},\pm}w_1, Im\, T^{*}_{\widetilde{\varphi},\pm}w_2) \st w_1,w_2\in C_{0}^{\infty}(\Sigma)\} .\] 
	Clearly, $D\subset L^2(\Sigma)\times  L^2(\Sigma) $ is a subspace. Now	for a given $v=(v_1+iv_2)\in L^2(\Sigma)$, let us  consider the linear functional 
	$L: D \to \mathbb{R}$, 
	\[L(Re\,T^{*}_{\widetilde{\varphi}, \pm}w_1, Im\,T^{*}_{\widetilde{\varphi}, \pm}w_2) = \langle w_1,  Re\,v\rangle_{L^2}+ \langle w_2,  Im\,v\rangle_{L^2}.\]
	By the above Carleman estimate (cf.\eqref{CAR}), it follows that 
	\begin{align*}|L(Re\,T^{*}_{\widetilde{\varphi}, \pm}w_1, Im\,T^{*}_{\widetilde{\varphi}, \pm}w_2)| &\leq ||w_1||_{L^2} ||v_1||_{L^2} +||w_2||_{L^2} ||v_2||_{L^2}\\[1mm]
	&\leq C\h^{-1}(\|Re\,T^{*}_{\widetilde{\varphi},\pm} w_1||_{L^2} ||v_1||_{L^2}+\|Im\,T^{*}_{\widetilde{\varphi},\pm} w_2||_{L^2} ||v_2||_{L^2}).
	\end{align*}
	The Hahn-Banach theorem ensures that there is a bounded linear functional $\widetilde L : L^2(\Sigma)\times  L^2(\Sigma) \to \mathbb{R}$ satisfying $\widetilde L = L$ on $D$ and $||\widetilde L|| \leq C\h^{-1} ||v||_{L^2(\Sigma)}$. 
	By the Riesz Representation theorem there exists $(u_1,u_2) \in L^2(\Sigma)\times  L^2(\Sigma) $ such that for all $(\theta_1,\theta_2)\in D $,
	\[	\widetilde L (\theta) = \langle u_1, \theta_1 \rangle_{L^2} + \langle u_2, \theta_2 \rangle_{L^2}, \quad \theta_1= Re\,T^{*}_{\widetilde{\varphi},\pm} w_1,\,  \theta_2= Re\,T^{*}_{\widetilde{\varphi},\pm} w_2,\]  
	this implies \begin{align*} \langle Re\,T_{\widetilde{\varphi},\pm} u_1, w_1 \rangle_{L^2} + \langle Im\,T_{\widetilde{\varphi},\pm} u_2, w_2 \rangle_{L^2}
	= \langle Re\, v, w_1\rangle_{L^2} +\langle Im\,v, w_2\rangle_{L^2}
	\end{align*}
	for all  $(w_1, w_2) \in C^{\infty}_0(\Sigma)\times C^{\infty}_0(\Sigma)$. 
	
	Hence, $Re\,T_{\widetilde{\varphi},\pm} u_1 =Re\,v$ in $\Sigma$, and $Im\,T_{\widetilde{\varphi},\pm} u_2 =Im\,v$ in $\Sigma$, along with the bound: $||u_k||_{L^2}\leq  C\h^{-1} ||v_k||_{L^2}$, $k=1,2$. 
		Therefore, for $v\in L^2(\Sigma)$, there exists $u=u_1+iu_2\in L^2(\Sigma),$
	$	\,T_{\widetilde{\varphi},\pm}u = v$  in $\Sigma$. 
	\hfill\end{proof}
Now we will apply the above result in order to establish the CGO solution for the equation 
\begin{equation}\label{me2} ((e_1+ ie_2)\cdot \nabla_{t,s})^2a_0 + c(t,s)a_0=0 \quad\mbox{in }\Sigma\subset\mathbb{R}^2,\end{equation}
where $c(t,s)$ is a bounded function in $\Sigma$. 
Let \begin{equation}\label{ansatz2}
a_0 = e^{\frac{\widetilde{\varphi}-i\widetilde{\psi}}{\h}}(b_0+\rho(\cdot;\h))
\end{equation}
where  
\begin{equation}\label{phipsi}
\begin{cases}
\widetilde{\varphi}(t,s)= at+bs\\[2pt] \widetilde{\psi}(t,s) = bt-as
\end{cases}\quad (a,b)\neq (0,0)
\end{equation}
are harmonic conjugate to each other, i.e. $\partial_t\widetilde{\psi} =\partial_s\widetilde{\varphi},$ and  $\partial_s\widetilde{\psi}=-\partial_t\widetilde{\varphi}$.  The amplitude function $b_0\in C^\infty(\overline{\Sigma})$ is non-zero, and the reminder $\rho$ to be determined later.

\vspace{5pt}
Now plugging \eqref{ansatz2} into \eqref{me2} we get
\begin{equation}\begin{aligned}\label{tau0}
&e^{-\frac{(at+bs)}{\h}}\tau^2\,((e_1+ ie_2)\cdot \nabla_{t,s})^2 e^{\frac{(at+bs)}{\h}}\big(e^{\frac{-i(bt-as)}{\h}}\rho(\cdot;\h)\big)  +\tau^2\, c(t,s)\big(e^{\frac{-i(bt-as)}{\h}}\rho(\cdot;\h)\big)\\[1mm]
& = - e^{-\frac{(at+bs)}{\h}}\tau^2\,\big((e_1+ ie_2)\cdot \nabla_{t,s}\big)^2 e^{\frac{(at+bs)-i(bt-as)}{\h}}b_0  +\tau^2\, c(t,s) e^{\frac{-i(bt-as)}{\h}}\,b_0.
\end{aligned}\end{equation}
Since 
\begin{equation}\begin{aligned}\label{tau1}
&	e^{-\frac{(at+bs)-i(bt-as)}{\h}}\,\tau^2\,((e_1+ ie_2)\cdot \nabla_{t,s})^2 e^{\frac{(at+bs)-i(bt-as)}{\h}}b_0 \\
&	= b_0\left[(e_1+ ie_2)\cdot \nabla_{t,s}\big((at+bs)-i(bt-as)\big)\right]^2
+ \h b_0\left((e_1+ie_2)\cdot\nabla_{t,s}\right)^2\big((at+bs)-i(bt-as)\big)\\
&\quad	+ 2\h \big[(e_1+ ie_2)\cdot \nabla_{t,s} b_0\big]\,\big[(e_1+ ie_2)\cdot\nabla_{t,s}\big((at+bs)-i(bt-as)\big)\big]
+\h^2((e_1+ ie_2)\cdot \nabla_{t,s})^2b_0,
\end{aligned}\end{equation}
where due to our choice of $\widetilde{\varphi}=(at+bs)$ and $\widetilde{\psi}=(bt-as)$ it turns out
\begin{align*}
\begin{cases}\big[(e_1+ ie_2)\cdot \nabla_{t,s}\big((at+bs)-i(bt-as)\big)\big]=(a-ib)+i(b+ia) =0,\\[2mm]
\left((e_1+ie_2)\cdot\nabla_{t,s}\right)^2\big((at+bs)-i(bt-as)\big) =0.
\end{cases}
\end{align*}
So the $\mathcal{O}(\tau^\alpha)$ terms for $\alpha=0,1$ in \eqref{tau1} are identically $0$.  
\vspace{5pt}

Thus from \eqref{tau1} and \eqref{tau0}, we get 
\begin{equation}\begin{aligned}\label{tau2}
&e^{-\frac{(at+bs)}{\h}}\tau^2\,((e_1+ ie_2)\cdot \nabla_{t,s})^2 e^{\frac{(at+bs)}{\h}}\underbrace{\big(e^{\frac{-i(bt-as)}{\h}}\rho(\cdot;\h)\big)}_{u\in L^2(\Sigma)}  +\tau^2\, c(t,s)\underbrace{\big(e^{\frac{-i(bt-as)}{\h}}\rho(\cdot;\h)\big)}_{u\in L^2(\Sigma)}\\[1mm]
& =  \underbrace{\tau^2\,\,e^{\frac{-i(t-s)}{\h}}\big[((e_1+ ie_2)\cdot \nabla_{t,s})^2b_0 + c\,b_0\big]}_{v\in L^2(\Sigma)}.
\end{aligned}\end{equation}
Now by Proposition \ref{ex_le},  the above equation \eqref{tau2} is solvable in $L^2(\Sigma)$. Hence there exists a $\rho\in L^2(\Sigma)$ with satisfying the estimate (cf. \eqref{tau_esti}): $\tau \|\rho\|_{L^2(\Sigma)} \leq \tau^2\, C$, where $C>0$ independent of $\tau>0$ small enough, or, 
\begin{equation}\label{tau_rho}
\|\rho\|_{L^2(\Sigma)} =\mathcal{O}(\tau), \qquad 0<\tau\ll 1. 
\end{equation}

\vspace{2pt}
\noindent
This establishes the CGO solution (cf.\eqref{ansatz2}) solving the transport equation \eqref{me2}, with the required decay estimate on the reminder term (cf.\eqref{tau_rho}). 

\vspace{5pt}
\noindent
Consequently, we have established the CGO solutions of the transport equation \eqref{nte1}, or \eqref{transportequation1}.

\begin{remark}[Regularity regarding \eqref{me2}]\label{reg}
It is evident from \eqref{tau2} that $((e_1+ ie_2)\cdot \nabla_{t,s})^2\rho(\cdot;\h)=\mathcal{O}(1)$ in $L^2(\Sigma)$. Since $((e_1+ ie_2)\cdot \nabla_{t,s})^{-1}:L^2(\Sigma)\mapsto L^2(\Sigma)$ is a continuous operator, so in particular $((e_1+ ie_2)\cdot \nabla_{t,s})\rho(\cdot;\h)\in L^2(\Sigma)$ and $\|((e_1+ ie_2)\cdot \nabla_{t,s})\rho(\cdot;\h)\|_{L^2}=\mathcal{O}(1)$.
We conclude that, $a_0$ (cf. \eqref{ansatz2}) solving \eqref{me2} is in $H^2(\Sigma)$ in a sense that $((e_1+ ie_2)\cdot \nabla_{t,s})^ka_0\in L^2(\Sigma)$ for $k=0,1,2$.
\end{remark}

Next, we would like to discuss the solvability of the transport equation \eqref{transportequation2}, i.e. 
\begin{equation}\label{nte2}\begin{aligned}
	&4((\mu_1+i\mu_2)\cdot\nabla)^2 a_1 + A (\mu_1+\I\mu_2)\cdot(\mu_1+\I\mu_2)\, a_1 \\
&= -4\big(((\mu_1+i\mu_2)\cdot\nabla)\circ \Delta\big) a_0 - \left(B\cdot(\mu_1+\I\mu_2)\right)a_0 \quad\mbox{in }\Omega
\end{aligned}\end{equation}
Comparing with \eqref{nte1}, the equation \eqref{nte2} stands as the non-homogeneous equation. In particular, we will be interested in solving 
\begin{equation}\label{me3} ((e_1+ ie_2)\cdot \nabla_{t,s})^2a_1 + c(t,s)a_1= f \quad\mbox{in }\Sigma\subset\mathbb{R}^2\end{equation}
where $f\in L^2(\Sigma)$ is some given non-zero function.

We can use the same methodology as before to prove the existence of the solution. Let us start with an ansatz $a_1 = e^{\frac{\widetilde{\vp}-\I\widetilde{\psi}}{\hh}} \left(b_1(x) + \rho_1(x,\hh)\right)$, where $0<\delta<\tau_1<1$ for some fixed $\delta>0$ small enough,  $\widetilde{\varphi},\widetilde{\psi}$ are as in \eqref{phipsi}, $b_1\in C^\infty(\overline{\Omega})$ is non-zero, and $\rho_1$ to be determined. Then substituting it in the equation \eqref{me3} we get
\begin{equation}\label{key_3}\begin{aligned}
&	e^{-\frac{\widetilde{\varphi}}{\hh}}((e_1+ ie_2)\cdot \hh\nabla_{t,s})^2	e^{\frac{\widetilde{\varphi}}{\hh}} \,(e^{-\frac{i\widetilde{\psi}}{\hh}}\rho_1)  + \hh^2 c(t,s)\, (e^{-\frac{i\widetilde{\psi}}{\hh}}\rho_1)\\
&	= -\hh^2 \,e^{\frac{-i(t-s)}{\h}}\big[((e_1+ ie_2)\cdot \nabla_{t,s})^2b_1 + c\,b_1\big] + f
\end{aligned}
\end{equation}
Using Proposition \ref{ex_le} we get $\rho_1\in L^2(\Sigma)$ solving \eqref{key_3} with $\|\rho_1\|_{L^2(\Sigma)}=\mathcal{O}(1)$ (depending on fixed $\delta>0$). Essentially, by the above Remark \ref{reg}, we can also conclude that the above equation \eqref{me3} is solvable in $H^2(\Sigma)$ in a sense that $((e_1+ ie_2)\cdot \nabla_{t,s})^ka_1\in L^2(\Sigma)$ for $k=0,1,2$. 

\vspace{3pt}
\noindent
Consequently, one establishes the existence of the solution of the transport equation \eqref{nte2}.
\begin{remark}[Regularity regarding \eqref{transportequation1}-\eqref{transportequation2}]\label{reg2}
Here we mention the regularity of $a_0$, $a_1$ solving \eqref{nte1} (or, \eqref{transportequation1}), \eqref{nte2} (or, \eqref{transportequation2}) respectively. Note that, we need $a_0, a_1 $ to be in $H^4(\O)$ (see Proposition \ref{solvibility}, Remark \ref{Rem_adjoint_op}). Now it is easy to see, for example the regularity of $a_0$ directly depends on the regularity of $A\in W^{3,\infty}(\Omega)$. Since $\partial_{x_l}a_0$, $l=1,..,n$ solves 
\begin{equation*}
4((\mu_1+ i\mu_2)\cdot \nabla)^2 \partial_{x_l}a_0 + A(x)(\mu_1+ i\mu_2)\cdot (\mu_1+ i\mu_2)\partial_{x_l}a_0 =-\partial_{x_l}A(x)(\mu_1+ i\mu_2)\cdot (\mu_1+ i\mu_2)a_0
\end{equation*}
hence we find $\partial_{x_l}a_0 \in H^2(\Omega)$ in a sense that $((\mu_1+ i\mu_2)\cdot \nabla)^k \partial_{x_l}a_0\in L^2(\Omega)$ for $k=0,1,2$, and repeating this process, our requirement of $a_0\in H^4(\O)$ would be fulfilled. In fact, $A\in  W^{3,\infty}(\Omega)$ gives $a_0\in H^5(\Omega)$.  Similarly, we conclude the same for $a_1$ to be in $H^4(\Omega)$, and to get that regularity, we use $a_0\in H^5(\Omega)$ and $B\in W^{2,\infty}(\Omega)$, as per right hand side of \eqref{nte2} appears in that form. This justifies our extra regularity assumption on the coefficients.   
\end{remark}
Summing up, we have 
\begin{proposition}\label{solvibility2}
	Let $\Omega\subset\mathbb{R}^n$ be a bounded domain. Let $\wt{A}, A^{\sharp}	\in W^{3,\infty}(\mathbb{R}^n,\mathbb{C}^{n^{2}}) \cap \mathcal{E}^{\prime}(\overline{\Omega})$ and consider the transport equations
	\begin{equation}\label{tr5}
	\begin{cases}
	4((\mu_1+i\mu_2)\cdot\nabla)^2\wt{a}_{0} +\wt{A}(x)(\mu_1+i\mu_2)\cdot(\mu_1+i\mu_2)\wt{a}_{0}=0 \mbox{ in }\Omega\\[4pt]
	4((-\mu_1+i\mu_2)\cdot\nabla)^2a^{\sharp}_{0} +A^{\sharp}(x)(-\mu_1+i\mu_2)\cdot(-\mu_1+i\mu_2)a^\sharp_{0}=0 \mbox{ in }\Omega.
	\end{cases}
	\end{equation}
	Then for all $\h >0$ small enough, there exists solutions $\widetilde{a}_0, a_0\in H^{4}(\Omega)$ of \eqref{tr5}  of the form
	\begin{equation}\label{aa}\begin{cases} 
	\wt{a}_0(x,\h) &= e^{\frac{(\widetilde{\vp}(x) - i\widetilde{\psi}(x))}{\h}}(\wt{b}(x) + \wt{\rho}(x;\h))\\[4pt]
	a^\sharp_0(x,\h) &= e^{\frac{(-\widetilde{\vp}(x) - i\widetilde{\psi}(x))}{\h}}(b^\sharp(x) + \rho^\sharp(x;\h))
	\end{cases}
	\end{equation}
	where $\widetilde{\varphi}$, $\widetilde{\psi}$ are real valued linear functions  as
	 \begin{equation*}\begin{cases}
	 \wt{\varphi}= a( \mu_1\cdot x) +b (\mu_2\cdot x)\\[3pt]
	 \wt{\psi}=b(\mu_1\cdot x)-a(\mu_2\cdot x)\end{cases}
	  \quad (a,b)\neq (0,0) \end{equation*}
	   satisfying 
	  \begin{equation}\label{Psi}
	  (\pm\mu_1+ i\mu_2)\cdot\nabla(\pm\widetilde{\vp}-i\widetilde{\psi})=0.
	  \end{equation}
	   The amplitude functions $\wt{b}, b^\sharp\in C^\infty(\overline{\O})$ are non-zero,
	and $\widetilde{\rho},\rho^\sharp\in H^4(\Omega)$ satisfies the estimate 
	\begin{equation}\label{rhoesti}
	\begin{cases}
\lVert \wt{\rho}(\cdot;\tau)\rVert_{L^2}, \lVert \rho^\sharp(\cdot;\tau)\rVert_{L^2}   = \Oc(\h)\\[1mm]
\|(\mu_1+ i\mu_2)\cdot\nabla \wt{\rho}(\cdot;\tau)\|_{L^2},  \|(\mu_1+ i\mu_2)\cdot\nabla \rho^{\#}(\cdot;\tau)\|_{L^2}=\mathcal{O}(1).
\end{cases}\quad 0<\tau\ll 1\end{equation}
\end{proposition}
Next we present the result of solving the transport equations \eqref{transportequation2}. 
\begin{proposition}\label{solvibility3}
	Let $a_0(x) \in H^4(\O)$, $\wt{A}, A^{\sharp} \in W^{3,\infty}(\O:\mathbb{C}^{n^2})$ and $\wt{B}, B^{\sharp} \in W^{2,\infty}(\O:\mathbb{C}^{n})$.
	Then there exist $\wt{a}_1(x),a^\sharp_1(x) \in H^4(\O)$ satisfying
	\begin{equation}\label{transport_1}
	\begin{cases}
	&	4((\mu_1+i\mu_2)\cdot\nabla)^2 \wt{a}_1 + \wt{A} (\mu_1+\I\mu_2)\cdot(\mu_1+\I\mu_2)\, \wt{a}_1 \\
	&\qquad\qquad\qquad=  -4\big(((\mu_1+i\mu_2)\cdot\nabla)\circ \Delta\big) \wt{a}_0 - \left(B\cdot(\mu_1+\I\mu_2)\right)\wt{a}_0\quad\mbox{in }\Omega\\[6pt]
	&4((-\mu_1+i\mu_2)\cdot\nabla)^2 a^\sharp_1 +  A^{\sharp} (-\mu_1+\I\mu_2)(-\mu_1+\I\mu_2)\, a^\sharp_1 \\
	&\qquad\qquad\qquad=  -4\big(((-\mu_1+i\mu_2)\cdot\nabla)\circ \Delta\big) a^\sharp_0 - \left(B\cdot(-\mu_1+\I\mu_2)\right)a^\sharp_0\quad\mbox{in }\Omega.
	\end{cases}
	\end{equation}
\end{proposition}
\section{Determination of the coefficients}\label{determination}
In this section we use the special form of the solution $u$ to determine the coefficients $A$, $B$ and $q$ in $\O$.
We consider two sets of parameters $A, \wt{A} \in W^{3,\infty}(\Rb^n:\mathbb{C}^{n^2})\cap \mathcal{E}'(\overline{\Omega})$, $B,\wt{B} \in W^{2,\infty}(\Rb^n:\mathbb{C}^{n^2})\cap \mathcal{E}'(\overline{\Omega})$ and $q,\tilde{q} \in L^{\infty}(\Omega)$.
Let $\mathcal{C}_{A,B,q}$ be the set of Cauchy data, given in \eqref{qw19}, corresponding to the operators $\Lc_{A,B,q}$. In this section we will show that if $\mathcal{C}_{A,B,q} = \mathcal{C}_{\wt{A},\wt{B},\wt{q}}$ on $\partial\O$ then $A=\wt{A}$, $B=\wt{B}$ and $q=\wt{q}$ in $\Omega$.

First let us extend our problem to a larger simply connected domain $\wt{\O}$.
Let $\wt{\O}\subset \Rb^n$ be a smooth, bounded, simply connected domain such that $\O \ssubset \wt{\O}$. Let us extend $q, \widetilde{q}$ as zero on $\wt{\O}\setminus \overline{\O}$ to  have $q,\wt{q} \in L^{\infty}(\wt{\O})$ with compact support inside $\wt{\O}$. And since by our assumption  $A, \wt{A} \in W^{3,\infty}(\Rb^n:\mathbb{C}^{n^2})\cap \mathcal{E}'(\overline{\Omega})$, $B,\wt{B} \in W^{2,\infty}(\Rb^n:\mathbb{C}^{n^2})\cap \mathcal{E}'(\overline{\Omega})$, so we can think $A,\wt{A} \in W^{3,\infty}(\wt{\O})$, $B,\wt{B}\in W^{2,\infty}(\wt{\O})$ with compact support inside $\wt{\O}$.
We extend the operator $\Lc$ over $\wt{\O}$ and denote them by the same notation.
\begin{proposition}\label{qw20}
	Let $\Omega \ssubset \wt{\O}$ be two bounded domains in $\mathbb{R}^n$ with smooth boundaries, and let $A,\wt{A} \in W^{3,\infty}(\wt{\O},\mathbb{C}^{n^{2}})$, $B,\wt{B} \in W^{2,\infty}(\wt{\O},\mathbb{C}^{n})$ and  $q,\, \wt{q} \in L^\infty(\O,\mathbb{C})$ satisfy $A_{jk} = \wt{A}_{jk}$, $B_j=\wt{B}_j$ and $q = \wt{q}$ in $\wt{\O}\setminus\O$ for all $j,k=1,\dots,n$.
	If the Cauchy data set (cf. \eqref{qw19}) $\mathcal{C}_{A,B,q}(\O) = \mathcal{C}_{\wt{A},\wt{B},\wt{q}}(\O)$, then $\mathcal{C}_{A,B,q}(\wt{\O}) = \mathcal{C}_{\wt{A},\wt{B},\wt{q}}(\wt{\O})$.
\end{proposition}
The proof of the above proposition is standard in the literature of Calder\'on type inverse problems and can be found in \cite{KRU1,SU}.

\subsection{Integral identity involving the coefficients}
We recall that
\[\Lc(x,D) =
(-\Delta)^2 + \sum_{j,k=1}^{n} A_{jk}(x)D^{j}D^{k} + \sum_{j=1}^{n} B_{j}(x)D^{j} + q(x),
\]
where $A_{jk}\in W^{3,\infty}(\wt{\Omega},\mathbb{C}^{n^{2}})$, $B_{j}\in W^{2,\infty}(\wt{\Omega},\mathbb{C}^n)$ and $q \in L^{\infty}(\wt{\Omega},\mathbb{C})$.

\noindent
We have the following integral identity 
\begin{equation}\label{integralidentity}
\int_{\wt{\Omega}} \left(\Lc(x,D)u\right)\overline{v} \D x - \int_{\wt{\Omega}} u\overline{\Lc^{*}(x,D)v}\D x = 0, \quad \forall u\in H^{4}_0(\wt{\O}), v\in H^{4}(\wt{\O}),
\end{equation}
where $H^{4}_0(\wt{\Omega})$ is the closure of $C_0^\infty(\wt{\Omega})$ functions in $H^{4}(\wt{\Omega})$ norm.
Let $u,\wt{u}\in H^{4}(\wt{\Omega})$ solve
\begin{equation}\begin{aligned}
\Lc_{A,B,q}(x,D)u &=0\quad\mbox{in }\wt{\Omega}\qquad \quad
\mbox{and}\quad
\Lc_{\wt{A},\wt{B},\wt{q}}(x,D)\widetilde{u} =0 \quad\mbox{in }\wt{\Omega},\\
\mbox{with }\quad (-\Delta)^l u|_{\partial\wt{\Omega}} &= (-\Delta)^l \widetilde{u}|_{\partial\wt{\Omega}},\quad \mbox{for }l=0,1.
\end{aligned}
\end{equation}
From the assumption of the Theorem \ref{mainresult} and Proposition \ref{qw20} on $\partial\wt{\O}$ we now get 
\begin{equation}
\partial_\nu(-\Delta )^{l}u = \partial_\nu(-\Delta )^{l}\widetilde{u},\qquad\mbox{for }l=0,1.
\end{equation}
So we have $(u-\widetilde{u})\in H^{4}_0(\wt{\Omega})$.

Let $v\in H^{4}(\wt{\Omega})$ satisfies $\Lc^{*}_{A,B,q}(x,D)v = 0 $ in $\wt{\Omega}$,
then from the integral identity \eqref{integralidentity} we get
\begin{equation}\begin{aligned}\label{integralidentityE}
&\left\langle \Lc_{A,B,q}(x,D)(\wt{u}-u),v\right\rangle_{L^2(\wt{\O})}\\
&=\int_{\wt{\Omega}}\lb \sum_{j,k=1}^{n}(A_{jk}-{\wt{A}_{jk}})D^{j}D^{k}\widetilde{u} + \sum_{j,k=1}^{n}(B_{j}-{\wt{B}_{j}})D^{j}\widetilde{u} + (q-\widetilde{q})\widetilde{u}\rb\overline{v}\,\D x
= 0
\end{aligned}\end{equation}
Next we choose $\wt{u}$ and $v$ to be the C.G.O. type solutions constructed in Section \ref{ccs}.
We choose $\varphi=\mu_1\cdot x$ and $\psi=\mu_2\cdot x$ for $\wt{u}$ and $\varphi=-\mu_1\cdot x$ and $\psi=\mu_2\cdot x$ for $v$, where  
$\mu_1,\mu_2 \in \mathbb{R}^n$ satisfying $|\mu_1|=|\mu_2|=1$ and $\mu_1\cdot\mu_2=0$. For $h>0$ small enough, we set the solutions are of the form
\begin{equation}
\label{CGO2}
\begin{cases}
\widetilde{u}(x)=e^{\frac{\mu_1\cdot x+ i\mu_2\cdot x}{h}} \lb\wt{a}_{0}(x,\mu_1,\mu_2)+h\wt{a}_{1}(x,\mu_1,\mu_2)+\widetilde{r}(x,\mu_1+ i\mu_2;h)\rb\quad&\mbox{in }\wt{\O},\\
v(x)=e^{\frac{-\mu_1\cdot x+i\mu_2\cdot x}{h}} \lb a^\sharp_{0}(x,-\mu_1,\mu_2)+a^\sharp_{1}(x,-\mu_1,\mu_2) +r^\sharp(x,-\mu_1+i\mu_2;h)\rb\quad&\mbox{in }\wt{\O}.
\end{cases}
\end{equation}
The amplitudes $\wt{a}_{l}(\cdot,\mu_1,\mu_2), a^{\sharp}_{l}(\cdot,-\mu_1,\mu_2)\in H^4({\wt{\O}})$, for $l=0,1$ satisfy the transport equations 
\begin{equation*}
\begin{aligned}
&\begin{cases}
4((\mu_1+i\mu_2)\cdot\nabla)^2\wt{a}_{0} +\wt{A}(x)(\mu_1+i\mu_2)\cdot(\mu_1+i\mu_2)\wt{a}_{0}=0\quad\mbox{in }\Omega,  \\[4pt]
4((-\mu_1+i\mu_2)\cdot\nabla)^2a^\sharp_{0} +A^{\sharp}(x)(-\mu_1+i\mu_2)\cdot(-\mu_1+i\mu_2)a^\sharp_{0}=0\quad\mbox{in }\Omega,
\end{cases}
\\[2mm]
&\begin{cases}
&	4((\mu_1+i\mu_2)\cdot\nabla)^2 \wt{a}_1 + \widetilde{A} (\mu_1+\I\mu_2)\cdot(\mu_1+\I\mu_2)\, \wt{a}_1 \\
&\qquad\qquad\qquad=  -4\big(((\mu_1+i\mu_2)\cdot\nabla)\circ \Delta\big) \wt{a}_0 - \left(\wt{B}\cdot(\mu_1+\I\mu_2)\right)\wt{a}_0\quad\mbox{in }\Omega,\\[6pt]
&4((-\mu_1+i\mu_2)\cdot\nabla)^2 a^\sharp_1 +  A^{\sharp} (-\mu_1+\I\mu_2)(-\mu_1+\I\mu_2)\, a^\sharp_1 \\
&\qquad\qquad\qquad=  -4\big(((-\mu_1+i\mu_2)\cdot\nabla)\circ \Delta\big) a^\sharp_0 - \left(B^\sharp \cdot(-\mu_1+\I\mu_2)\right)a^\sharp_0\quad\mbox{in }\Omega.
\end{cases}
\end{aligned}
\end{equation*}
along with 
\begin{equation*}
\label{eq_remainder_r_j}
\|\wt{r}\|_{H^{4}_{\textrm{scl}}},\, \|r^\sharp\|_{H^{4}_{\textrm{scl}}}=\mathcal{O}(h^2).
\end{equation*}

Now substituting \eqref{CGO2} in \eqref{integralidentityE} we get
\begin{equation}\label{CGOIdentityE}
\begin{aligned}
0=&\sum_{j,k=1}^{n} \frac{-1}{h^2} \int_{\wt{\Omega}} (A_{jk}-{\wt{A}_{jk}})(\mu_1+ i\mu_2)_{j}(\mu_1+ i\mu_2)_{k}
\lb \wt{a}_{0} +h\wt{a}_1+ \wt{r}\rb\overline{\lb a^\sharp_{0}+ha^\sharp_1 +r^\sharp\rb }\ \D x\\
&+\sum_{j,k=1}^{n} \frac{-i}{h}\int_{\wt{\Omega}} (A_{jk}-{\wt{A}_{jk}})(\mu_1+ i\mu_2)_{j}
D^{k}\lb \wt{a}_{0} +h\wt{a}_1+ \wt{r}\rb\overline{\lb a^\sharp_{0}+ha^\sharp_1 +r^\sharp\rb }\ \D x\\
&+\sum_{j,k=1}^{n} \int_{\wt{\Omega}} (A_{jk}-{\wt{A}_{jk}})
\lb D^{j}D^{k} (\wt{a}_{0}+h\wt{a}_1+\wt{r})\rb\overline{\lb a^\sharp_{0}+ha^\sharp_1 +r^\sharp\rb }\ \D x\\
&+\sum_{j=1}^{n} \frac{-i}{h}\int_{\wt{\Omega}} (B_{j}-{\wt{B}_j})(\mu_1+ i\mu_2)_{j}
\lb \wt{a}_{0}+h\wt{a}_1+\wt{r}\rb\overline{\lb a^\sharp_{0}+ha^\sharp_1 +r^\sharp\rb }\ \D x\\
&+\sum_{j=1}^{n} \int_{\wt{\Omega}} (B_{j}-{\wt{B}_{j}})\left(D^{j}
\lb \wt{a}_{0} +h\wt{a}_1 + \wt{r}\rb\right)\overline{\lb a^\sharp_{0}+ha^\sharp_1 +r^\sharp\rb }\ \D x \\
&+\int_{\wt{\Omega}} (q-\widetilde{q})\lb \wt{a}_{0} + h\wt{a}_1
+\wt{r}\rb \overline{\lb a^\sharp_{0}+ha^\sharp_1 +r^\sharp\rb }\ \D x,
\end{aligned}
\end{equation}
where $(\mu_1+i \mu_2)_l$ is the $l$-th component of the vector $(\mu_1+i\mu_2) \in \Cb^n$.
We assume that $\left(A-\wt{A}\right)$ is not an isotropic matrix. If it is isotropic, then the first term in \eqref{CGOIdentityE} vanishes immediately. Multiplying \eqref{CGOIdentityE} by $h^{2}$ and letting $h\to 0$ we get
\begin{equation}\label{iad}
\begin{aligned}
\sum_{j,k=1}^{n} \int_{\wt{\Omega}} \left(A_{jk}-{\wt{A}_{jk}}\right)(\mu_1+ i\mu_2)_{j}(\mu_1+ i\mu_2)_{k}
\,\wt{a}_{0}\,\overline{a^\sharp_{0}}\ \D x=0.
\end{aligned}
\end{equation}
This follows from the fact that $A,\wt{A}\in W^{3,\infty}(\wt{\Omega})$; $B,\wt{B}\in W^{2,\infty}(\wt{\Omega})$, $\wt{a}_0,\, a^\sharp_0,\, \wt{a}_1,\, a^\sharp_1\in H^4(\overline{\wt{\O}})$, and the fact $\|\wt{r}\|_{H^{4}_{\textrm{scl}}},\, \|r^\sharp\|_{H^{4}_{\textrm{scl}}}
=\mathcal{O}(h^2)$. We use the later fact to obtain $\|\wt{r}\|_{L^2},\, \|r^\sharp\|_{L^2} =\mathcal{O}(h^2)$, $\|D^{\beta}\wt{r}\|_{L^2},\, \|D^{\beta}r^\sharp\|_{L^2}=\mathcal{O}(h)$, for $|\beta|=1$ and $\| D^{\A}\wt{r}\|_{L^2},\, \|D^{\A}r^\sharp\|_{L^2}=\mathcal{O}(1)$, for $|\A|=2$.
\subsection*{Determining the difference $(A-\widetilde{A})$ up-to isotropic matrix}
A priori we do not assume $\left(A-\wt{A}\right)$ be an isotropic matrix, but then we show here that the identity \eqref{iad} forces the difference to be isotropic. Let us  
begin with the identity \eqref{iad}.
We plug the expression of $\wt{a}_0$ and $a^\sharp_0$ given in \eqref{aa}, to find 
\begin{equation*}
\sum_{j,k=1}^{n}\int_{\wt{\Omega}} \left(A_{jk}-{\wt{A}_{jk}}\right)(\mu_1+ i\mu_2)_{j}(\mu_1+ i\mu_2)_{k}
\,(\widetilde{b}(x) + \widetilde{\rho}(x;\h))(\overline{b^\sharp}(x) + \overline{\rho^\sharp}(x;\h))\, \D x=0.
\end{equation*}
Then by taking $\h\to 0$ and using  $\lVert \wt{\rho}\rVert_{L^2}, \lVert \rho^\sharp\rVert_{L^2}   = \Oc(\h)$,  we obtain the limiting identity as 
\begin{equation}\label{limitingidentityE_1}
\sum_{j,k=1}^{n} \int_{\wt{\Omega}} \left(A_{jk}-{\wt{A}_{jk}}\right)(\mu_1+ i\mu_2)_{j}(\mu_1+ i\mu_2)_{k}
\,\widetilde{b}(x)\,\overline{b^\sharp}(x)\, \D x=0,
\end{equation}
where $\wt{b}\in C^\infty(\overline{\Omega})$ and $b^\sharp\in C^\infty(\overline{\Omega})$  are non-zero complex amplitude functions.  
\begin{remark}
In particular, we will be choosing $\wt{b}, b^\sharp\in C^\infty(\overline{\Omega})$ as the non-zero solutions of the homogeneous equations:
\[\begin{cases}\big((\mu_1+i\mu_2)\cdot\nabla\big)^2 \widetilde{b}= 0\quad\mbox{ in }\Omega\\[4pt]
\big((-\mu_1+i\mu_2)\cdot\nabla\big)^2b^\sharp= 0\quad\mbox{ in }\Omega.
\end{cases}
\]
\end{remark}
Now let us choose $b^\sharp=1$ and $\widetilde{b}=e^{-ix\cdot\xi}$ where $\mu_1\perp\mu_2\perp\xi$. 
So, from the above identity we obtain
\begin{equation}\label{fourier}
\sum_{j,k=1}^{n} \int_{\wt{\Omega}} \left(A_{jk}-{\wt{A}_{jk}}\right)(\mu_1+ i\mu_2)_{j}(\mu_1+ i\mu_2)_{k}
\,e^{-ix\cdot\xi}\, \D x=0. 
\end{equation} 
The above identity holds for all non-zero vectors $\mu_1$, $\mu_2$, $\xi$ in $\mathbb{R}^n$, where $|\mu_1|=|\mu_2|$ and $\mu_1\perp\mu_2\perp\xi$.

Let us recall that $A,\wt{A},B,\wt{B},q,\wt{q}$ zero on $\wt{\O}\setminus\overline{\O}$. We further extend $A,\wt{A},B,\wt{B},q,\wt{q}$ by $0$ outside $\wt{\O}$ to all over  $\mathbb{R}^n$. Then \eqref{fourier} reads
\begin{equation}\label{fourier_1}
\sum_{j,k=1}^{n} \int_{\Rb^n} \left(A_{jk}-{\wt{A}_{jk}}\right)(\mu_1+ i\mu_2)_{j}(\mu_1+ i\mu_2)_{k}\,e^{-ix\cdot\xi}\, \D x=0. 
\end{equation}
To this end we closely follow the arguments in \cite{BG19}.
Let us fix $\xi \in \Rb \setminus \{0\}$ and consider the orthonormal basis $\Bb$ of $\Rb^n$ as
\[	\Bb := \Big\{\mu_1,\mu_2,\dots,\mu_{n-1},\frac{\xi}{|\xi|}\Big\}.
\]
Following \cite{BG19} we have a unique decomposition of the symmetric 2-tensor field $(A-\widetilde{A})$ in $\tilde{\Omega}$ as
\begin{equation}\label{A-tildeA}
(A - \widetilde{A}) = F +dV,\quad \mbox{where }\sum_{j=1}^{n} \partial_{x_j}F_{jk} = 0, \quad \forall\ k = 1,2,\dots,n,
\end{equation}
and $V=(V_1,V_2,\dots,V_n)$ is a smooth 1-form in $\wt{\Omega}$ with $V|_{\partial\Omega} = 0$.
Substituting the form of $(A-\wt{A})$ in \eqref{fourier_1} and using integration by parts we directly obtain
\begin{equation}\label{fourier_2}
\sum_{j,k=1}^{n}\int_{\Rb^n} F_{jk}(\mu_p+ i\mu_l)_{j}(\mu_p+ i\mu_l)_{k}
\,e^{-ix\cdot\xi}\, \D x=0, 
\end{equation}
for $\mu_p,\mu_l \in \Bb$, $l,p=1,\dots,n-1$ and $p\neq l$. 
Therefore, we get
\begin{equation}\label{eigen_1}
\widehat{F}_{jk}(\xi)(\mu_p+ i\mu_l)_{j}(\mu_p+ i\mu_l)_{k} = 0\quad  \mbox{for } p,l=1,\dots,n-1, \quad p\neq l.
\end{equation}
Replacing $\mu_l$ by $-\mu_l$ in $\Bb$ we get
\begin{equation}\label{eigen_2}
\widehat{F}_{jk}(\xi)(\mu_p-i\mu_l)_{j}(\mu_p-i\mu_l)_{k} = 0\quad  \mbox{for } p,l=1,\dots,n-1, \quad p\neq l.
\end{equation}
From \eqref{eigen_1} and \eqref{eigen_2} we directly obtain
\begin{equation}\label{eigen_3}
\begin{cases}
\langle\widehat{F}(\xi)\mu_1,\mu_1\rangle = \langle\widehat{F}(\xi)\mu_l,\mu_l\rangle, \quad &\mbox{for }l=2,3,\dots,n-1,\\
\langle\widehat{F}(\xi)\mu_p,\mu_l\rangle = 0 \quad &\mbox{for }l,p=2,3,\dots,n-1, \quad l\neq p,\\
\sum_{j=1}^{n} \widehat{F}_{jk}\xi_j = 0, \quad &\mbox{for } k=1,2,\dots,n.
\end{cases}
\end{equation}
Writing $d(\xi) = \langle\widehat{F}(\xi)\mu_1,\mu_1\rangle$ we see
$\wh{F}(\xi) = P^t D P$, where $D = diag(d(\xi),d(\xi),\dots,d(\xi),0)$ and
\[	P^t = \left( \mu_1 \quad \mu_2 \quad \dots \quad \mu_{n-1}\quad \frac{\xi}{|\xi|}\right).
\]
Using this we can write a formal expression of $\wh{F}(\xi)$ for $\xi\in \Rb^n\setminus\{0\}$ as
\begin{equation}\label{expression}
\wh{F}(\xi) = d(\xi) \left(I - \frac{\xi\otimes\xi}{|\xi|^2}\right), \quad \xi \in \Rb^n\setminus\{0\}.
\end{equation}
Consequently, in $x$-variable we get
\[	F_{jk}(x) = d_{\#}(x)\delta_{jk} + R_jR_k(d_{\#}(x)),
\]
where $d_{\#} \in L^2(\Rb^n)$ with $\wh{d_{\#}}(\xi) = d(\xi)$ and $R_j$ are the classical Riesz transformation defined as $\wh{R_jf}(\xi)=\frac{1}{i}\frac{\xi_j}{|\xi|}\wh{f}(\xi)$, for $f \in L^2(\Rb^n)$. Observe that $d_{\#}$ is compactly supported in $\O$, (see \cite[Equation 2.34]{BG19}). Let $\wt{d}(x) \in H^2_0(\Rb^n)$ solving $-\Delta \wt{d} = d_{\#}$ in $\Rb^n$. Then by using standard definition of Riesz transform we get
\[	F_{jk}(x) = d_{\#}(x)\delta_{jk} + \frac{1}{2}\left[\frac{\partial}{\partial x_j}\left(\frac{\partial\wt{d}}{\partial x_k}\right)+\frac{\partial}{\partial x_k}\left(\frac{\partial\wt{d}}{\partial x_j}\right)\right], \quad \mbox{in }\Rb^n.
\]
Therefore, from \eqref{A-tildeA} we get
\begin{equation}\label{A-tildeA_1}
\left(A-\wt{A}\right)_{jk} = d_{\#}(x)\delta_{jk} + \frac{1}{2}\left[\frac{\partial}{\partial x_j}\wt{V}_k+\frac{\partial}{\partial x_k}\wt{V}_j\right], \quad \mbox{in }\Rb^n,
\end{equation}
where $\wt{V} = \nabla_x\wt{d} + V \in H^1_0(\wt{\O})$.

Substituting this form of $(A-\widetilde{A})$ back in \eqref{limitingidentityE_1} and using the fact that $(\mu_1+i\mu_2)\cdot(\mu_1+i\mu_2)=0$ we get
\begin{equation*}
\sum_{j,k=1}^{n} \int_{\wt{\Omega}} \left[\frac{\partial}{\partial x_j}\wt{V}_k + \frac{\partial}{\partial x_k}\wt{V}_j\right](\mu_1+ i\mu_2)_{j}(\mu_1+ i\mu_2)_{k}
\,\widetilde{b}(x)\overline{b}(x)\, \D x=0.
\end{equation*}
Using integration by parts and the fact that $\wt{V}|_{\partial\wt{\Omega}}=0$ we obtain
\begin{equation*}
\sum_{j,k=1}^{n} \int_{\wt{\Omega}} \wt{V}_k (\mu_1+ i\mu_2)_{k}(\mu_1+ i\mu_2)_{j}\,\frac{\partial}{\partial x_j}\left(\widetilde{b}(x)\overline{b}(x)\right)\, \D x=0.
\end{equation*}
Now, we choose $\wt{b}(x) = e^{-ix\cdot\xi}(\mu_1\cdot x)$, $b^\sharp(x)=1$ in $\wt{\Omega}$ and see that
\begin{equation}\label{limitingidentityE_5.1}
\sum_{k=1}^{n} \int_{\wt{\Omega}} e^{-ix\cdot\xi}\,(\mu_1+ i\mu_2)_{k}\wt{V}_k(x)\, \D x=0.
\end{equation}
Observe that we can replace $\mu_2$ by $-\mu_2$ and using a similar analysis we obtain
\begin{equation}\label{limitingidentityE_5.2}
\sum_{k=1}^{n} \int_{\wt{\Omega}} e^{-ix\cdot\xi}\,(\mu_1- i\mu_2)_{k}\wt{V}_k(x)\, \D x=0.
\end{equation}
Adding \eqref{limitingidentityE_5.1} and \eqref{limitingidentityE_5.2} we get
\begin{equation}\label{limitingidentityE_5}
\int_{\wt{\Omega}} e^{-ix\cdot\xi}\,\left(\mu\cdot\wt{V}(x)\right)\, \D x=0, \quad \mbox{for all }\mu\in \Rb^n\setminus\{0\}\mbox{ perpendicular to }\xi.
\end{equation}
As $\wt{V}=0$ outside $\wt{\Omega}$ we can realise the above integration over $\Rb^n$.
Choosing $\mu=(-\xi_2,\xi_1,0,\dots,0)\in\Rb^n$ and evaluating the Fourier transform in \eqref{limitingidentityE_5} we see
\begin{equation}\label{V}
(d\wt{V})_{jk} := \partial_{x_k}\wt{V}_j - \partial_{x_j}\wt{V}_k = 0, \quad \mbox{in }\Rb^n,
\end{equation}
where $d$ is the exterior derivative. Having $\wt{\Omega}$ to be simply connected we obtain $p \in H^2(\wt{\Omega})$ such that
\begin{equation}\label{V_1}
\wt{V}(x) = \nabla p(x), \quad \mbox{in }\wt{\Omega}.
\end{equation}
Since, $\wt{V}|_{\wt{\Omega}}=0$ so we get $\nabla_{tan} p = \nabla p - (\partial_{\nu}p)\nu = 0$ on $\partial\wt{\Omega}$ and thus $p|_{\partial\wt{\Omega}}=c$ for some constant $c \in \Rb$.
Replacing $p$ by $p-c$ in $\wt{\Omega}$ we get $\wt{V} = \nabla p$ with $p|_{\partial\wt{\Omega}}=0=\partial_{\nu}p|_{\partial\wt{\Omega}}$. The normal derivative of the function $p$ vanishes on the boundary as a consequence of the fact the $\wt{V}=\nabla p$ vanishes on the boundary.
Therefore, summarizing the above analysis we get
\begin{equation}\label{key_2}
(A-\widetilde{A})_{jk} = d_{\#}(x)\delta_{jk} + \frac{\partial^2 p}{\partial x_j \partial x_k}, \quad \mbox{in }\wt{\Omega};\quad \mbox{with }d_{\#}\in L^2(\wt{\Omega}) \mbox{ and }p \in H^2_0(\wt{\Omega}).
\end{equation} 

Again going back to \eqref{limitingidentityE_1} and substituting the form of $(A-\widetilde{A})$ there, we get
\begin{equation}\label{A_1}
\sum_{j,k=1}^{n}\int_{\wt{\Omega}} (\mu_1+ i\mu_2)_j\frac{\partial^2 p}{\partial x_j\partial x_k} (\mu_1+ i\mu_2)_k
\,\wt{b}(x)\, \overline{b^{\sharp}(x)}\ \D x = 0.
\end{equation}
We get rid of the part of $(A-\widetilde{A})$ contributed by the function $d_{\#}$ by using the fact $(\mu_1+i\mu_2)\cdot(\mu_1+i\mu_2)=0$. Since $p|_{\partial\wt{\Omega}} = 0 = \partial_{\nu} p|_{\partial\wt{\Omega}}$, using integration by parts we obtain
\begin{equation}\label{A_2}
\int_{\wt{\Omega}} p(x)\, \left((\mu_1+ i\mu_2)\cdot\nabla\wt{b}(x)\right)\left((\mu_1+ i\mu_2)\cdot\nabla\overline{b^\sharp(x)}\right)\ \D x = 0.
\end{equation}
We choose $\wt{b}(x) = e^{-ix\cdot\xi}(\mu_1\cdot x)$, $b^\sharp(x) = \mu_1\cdot x$ in $\wt{\Omega}$ and obtain
\[	\int_{\wt{\Omega}} e^{-ix\cdot\xi} p(x)\, \D x = 0.
\]
Thus, varying $\xi$ we finally obtain $p=0$ and hence \begin{equation}\label{key}
(A-\widetilde{A})(x)=d_{\#}(x) I 
\end{equation}
Since $(A-\widetilde{A})\in W^{3,\infty}(\wt{\Omega})$ with $(A-\widetilde{A})=0$ in $\widetilde{\Omega}\setminus\overline{\Omega}$. So $d_{\#}\in W^{3,\infty}(\wt{\Omega})$ with $d_{\#}=0$ in $\widetilde{\Omega}\setminus\overline{\Omega}$. 
\vspace{1mm}
\subsection*{Determining the first order perturbation $B=\widetilde{B}$}
Writing $(A-\widetilde{A})=d_{\#}(x) I$ in \eqref{CGOIdentityE} we get
\begin{equation}\label{int_A_2}
\int_{\wt{\Omega}} d_{\#}(x)\left((\mu_1+ i\mu_2)\cdot\nabla
\wt{a}_{0}(x)\right)\, \overline{a^\sharp_0(x)}\ \D x 
+ \int_{\wt{\Omega}} \left(B-\widetilde{B}\right)\cdot(\mu_1+ i\mu_2)\,\wt{a}_{0}(x)\, \overline{a^\sharp_{0}(x)}\ \D x = 0.
\end{equation}
Next by substituting the form of the amplitudes in \eqref{int_A_2} we obtain
\begin{equation}\label{int_A_4}
\begin{aligned}
0=&	\int_{\wt{\Omega}} d_{\#}(x)\left((\mu_1+ i\mu_2)\cdot\nabla \wt{a}_{0}(x)\right)\, \overline{a^\sharp_0(x)}\ \D x +\int_{\wt{\Omega}} \left(B-\widetilde{B}\right)\cdot(\mu_1+ i\mu_2)\,\wt{a}_{0}(x)\, \overline{a^\sharp_0(x)}\ \D x \\
=& \frac{1}{\h}\int_{\wt{\Omega}} d_{\#}(x)\left[(\mu_1+ i\mu_2)\cdot\nabla(\widetilde{\vp}-i\widetilde{\psi})\right]\, \left(\wt{b}(x)+\wt{\rho}(x;\tau)\right)\overline{\left(b^\sharp(x)+\rho^\sharp(x;\tau)\right)}\ \D x\\
&+\int_{\wt{\Omega}} d_{\#}(x)\left[(\mu_1+ i\mu_2)\cdot\nabla\wt{b}(x)\right]\, \overline{b^\sharp(x)}\ \D x
+ \int_{\wt{\Omega}} d_{\#}(x)\left[(\mu_1+ i\mu_2)\cdot\nabla\wt{b}(x)\right]\, \overline{\rho^\sharp(x;\tau)}\ \D x\\
&-\int_{\wt{\Omega}} \wt{\rho}(x;\tau)\,\left[(\mu_1+ i\mu_2)\cdot\nabla (d_{\#}(x)\overline{b^\sharp(x)})\right] \ \D x
+ \int_{\wt{\Omega}} d_{\#}(x)\left[(\mu_1+ i\mu_2)\cdot\nabla\wt{\rho}(x;\tau)\right]\, \overline{\rho^\sharp(x;\tau)}\ \D x \\
&  +\int_{\wt{\Omega}} \left(B-\widetilde{B}\right)\cdot(\mu_1+ i\mu_2)\,\left(\wt{b}(x)+\wt{\rho}(x;\tau)\right)\overline{\left(b^\sharp(x)+\rho^\sharp(x;\tau)\right)}\ \D x.
\end{aligned}
\end{equation}
Note that, on the fourth integral in the right hand side of \eqref{int_A_4}, we did integration by-parts since $d_{\#}\overline{b^\sharp}\in H^1_0(\widetilde{\Omega})$. Since $(\mu_1+ i\mu_2)\cdot\nabla(\widetilde{\vp}-i\widetilde{\psi})=0$ (cf. \eqref{Psi}), so the first term in the rhs of \eqref{int_A_4} disappears. 
Recall that (cf. Proposition \ref{solvibility2}) we have $\lVert \wt{\rho}(\cdot;\tau)\rVert_{L^2(\O)}, \lVert \rho^\sharp(\cdot;\tau)\rVert_{L^2(\O)} = \mathcal{O}(\tau)$ and $\lVert (\mu_1+i\mu_2)\cdot \nabla_x \wt{\rho}(\cdot;\tau)\rVert_{L^2(\O)}= \mathcal{O}(1)$. 
So by taking $\tau\to 0$ in  \eqref{int_A_4} we obtain
\begin{equation}\label{int_A_5}
0=\int_{\wt{\Omega}} d_{\#}(x)\left[(\mu_1+ i\mu_2)\cdot\nabla\wt{b}(x)\right]\, \overline{b^\sharp(x)}\ \D x + \int_{\wt{\Omega}} \left(B-\widetilde{B}\right)\cdot(\mu_1+ i\mu_2)\,\wt{b}(x)\,\overline{b^\sharp(x)}\,\D x.
\end{equation}
Next, let us choose $\wt{b}(x)= e^{-ix\cdot\xi}$, since it implies $(\mu_1+i\mu_2)\cdot\nabla \wt{b} = 0$, so we obtain
\begin{equation}\label{int_B.1}
\int_{\wt{\Omega}} e^{-ix\cdot\xi} \left(B-\widetilde{B}\right)\cdot(\mu_1+ i\mu_2)\,\overline{b^\sharp(x)}\, \D x = 0,
\end{equation}
Replacing $\mu_2$ by $-\mu_2$ and doing the same analysis as before we obtain
\begin{equation}\label{int_B.2}
\int_{\wt{\Omega}} e^{-ix\cdot\xi} \left(B-\widetilde{B}\right)\cdot(\mu_1- i\mu_2)\,\overline{b^\sharp(x)}\, \D x = 0.
\end{equation}
Adding \eqref{int_B.1}, \eqref{int_B.2} and writing $\mathcal{B}=(B-\widetilde{B})$ we see
\begin{equation}\label{int_B}
\int_{\wt{\Omega}} e^{-ix\cdot\xi} \left(\mu\cdot\mathcal{B}\right)\,\overline{b^\sharp(x)}\, \D x = 0, \quad \mbox{for all }\mu \mbox{ parpendicular to }\xi.
\end{equation}
Let us take $b^\sharp(x)=1$. By extending $\mathcal{B}$ as $zero$ on $\Rb\setminus\overline{\wt{\Omega}}$ and choosing $\mu=(-\xi_2,\xi_1,0,\dots,0)$ we get
\[	(d\mathcal{B})_{jk} := \partial_{x_j}\mathcal{B}_k - \partial_{x_k}\mathcal{B}_j = 0, \quad \mbox{in }\wt{\Omega}, 
\]
where $d$ is the exterior derivative acting on the 1-form $\mathcal{B}$.
Using simply connectedness of $\wt{\Omega}$ we get $\Phi \in H^1_0(\wt{\Omega})$ such that $\mathcal{B}=\nabla\Phi$ in $\wt{\Omega}$.

Then plugging $\mathcal{B}=\nabla\Phi$ for $\Phi \in H^1_0(\wt{\Omega})$ in \eqref{int_B},  and then doing integration by parts, we obtain 
\begin{equation}\label{int_B2}
\int_{\wt{\Omega}} e^{-ix\cdot\xi}\,\Phi\,\left[(\mu_1+i\mu_2)\cdot\nabla\overline{b^\sharp(x)}\right]\, \D x = 0.
\end{equation}
Hence, by choosing $b^{\sharp}=-\mu_1\cdot x$ such that $(-\mu_1+i\mu_2)\cdot\nabla b^\sharp(x)=1$, from \eqref{int_B2} we obtain $\widehat{\Phi}(\xi)=0$ for all $\xi\in\mathbb{R}^n$, or, $\Phi\equiv0$.  Thus 
\begin{equation}\label{btb}B=\widetilde{B},\quad\mbox{ in }\widetilde{\Omega}.\end{equation}  
\subsection*{Determining the second order perturbation $A=\widetilde{A}$}
In \eqref{key}, we have already shown the difference $(A-\widetilde{A})=d_{\#}I$ in $\widetilde{\Omega}$, so it is remained to show $d_{\#}=0$. We get back to the identity \eqref{int_A_5} and put $B=\widetilde{B}$ (cf. \eqref{btb}) there and obtain  
\begin{equation}\label{int_A_20}
\int_{\wt{\Omega}} d_{\#}(x)\left[(\mu_1+ i\mu_2)\cdot\nabla\wt{b}(x)\right]\, \overline{b^\sharp(x)}\ \D x =0.
\end{equation}
Let us choose $\wt{b}(x)=(\mu_1\cdot x)e^{-ix\cdot\xi}$ and $b^\sharp=1$ in above, and we obtain $\widehat{d_{\#}}(\xi)=0$ for all $\xi\in\mathbb{R}^n$, or, $d_{\#}\equiv 0$. Thus 
\begin{equation}\label{ata}A=\widetilde{A},\quad\mbox{ in }\widetilde{\Omega}.\end{equation}  
\subsection*{Determining the potential $q=\widetilde{q}$} 
Let us put $A=\wt{A}$ and $B=\wt{B}$ in \eqref{CGOIdentityE} and observe that we end up with
\begin{equation}\label{int_q}
\int_{\wt{\Omega}} (q-\widetilde{q})\lb \wt{a}_{0}+h\wt{a}_1+
\wt{r}\rb \overline{\lb a^\sharp_{0}+ha^\sharp_1 +r^\sharp\rb }\ \D x = 0,
\end{equation}
where $\wt{a}_0$, $a_0$ satisfy \eqref{tr5}. As usual by
taking $h\to 0$, and then using the form of $\wt{a}_0, a_0$ in \eqref{aa} and then taking $\h \to 0$ we get 
\begin{equation}\label{int_q_1}
\int_{\wt{\Omega}} (q-\widetilde{q})\, \wt{b}(x)\,\overline{b^\sharp(x)}\ \D x = 0.
\end{equation}
Choosing $\wt{b}(x)=e^{-ix\cdot\xi}$, $b^\sharp(x)=1$, from \eqref{int_q_1} we obtain $\widehat{(q-\wt{q})}(\xi)=0$. Varying $\xi \in \Rb^n$ we finally obtain $q(x)=\wt{q}(x)$ in $\wt{\Omega}$.
Along with \eqref{btb} and \eqref{ata}, this completes the determination of $A=\widetilde{A}$, $B=\widetilde{B}$, and $q=\widetilde{q}$ in $\Omega$. The proof of the Theorem \ref{mainresult} is now complete. \qed


\end{document}